\documentclass[11pt,a4paper]{amsart}
\usepackage[utf8]{inputenc}
\usepackage[english]{babel}
\usepackage{amsmath}
\usepackage{amsfonts}
\usepackage{amssymb}
\usepackage{amsthm}
\usepackage{bbm}
\usepackage{float}

\usepackage{pgf}
\usepackage{tikz}
\usetikzlibrary{matrix,arrows.meta}
\usetikzlibrary{arrows,automata}
\usetikzlibrary{positioning}
\tikzset{
     arrow/.style = { very thick, ->, >=Triangle},
     darrow/.style = { very thick, <->, >=Triangle},
     dashow/.style = { very thick, dashed, ->, >=Triangle},
     state/.style={
           rectangle,
           rounded corners,
           draw=black, very thick,
           minimum height=3em,
           inner sep=6pt,
           text centered,
           },
}

\title{Inverting the coupling of the signed Gaussian free field with a loop-soup}
\date{}

\numberwithin{equation}{section}
\newtheorem{theorem}{Theorem}
\newtheorem*{theorem*}{Theorem}
\newtheorem{proposition}{Proposition}[section]
\newtheorem{corollary}[proposition]{Corollary}
\newtheorem{lemma}[proposition]{Lemma}
\newtheorem{remark}[proposition]{Remark}

\newtheorem{definition}{Definition}[section]
\usepackage[a4paper]{geometry}
\usepackage{hyperref}

\def\ccc{{\mathcal C}}
\def\R{{\mathbb R}}
\def\N{{\mathbb N}}
\def\Z{{\mathbb Z}}
\def\ooo{{\mathcal{O}}}
\def\E{{\mathbb{E}}}
\def\demi{{1\over 2}}
\def\eee{{\mathcal E}}

\def\aaa{\mathcal{A}}
\def\fff{{\mathcal F}}
\def\indic{{{\mathbbm 1}}}

\newcommand{\Pb}{\mathbb{P}}

\author{Titus Lupu \and Christophe Sabot \and Pierre Tarrès}
\address {
CNRS and Laboratoire de Probabilités, Statistique et Modélisation,
Sorbonne-Université,
75005 Paris,
France}
\email
{titus.lupu@upmc.fr}

\address {
Institut Camille Jordan,
Université Lyon 1, 
43 bd. du 11 nov. 1918,
69622 Villeurbanne cedex,
France}
\email
{sabot@math.univ-lyon1.fr}

\address {CNRS and Universit{\'e} Paris-Dauphine\\
PSL Research University, Ceremade, UMR 7534\\
Place du Mar{\'e}chal de Lattre de Tassigny\\
75775 Paris cedex 16, France\\
and \\
Courant Institute of Mathematical Sciences\\ 
New York, NYU-ECNU \\
Institute of Mathematical Sciences at NYU Shanghai}
\email
{tarres@nyu.edu}

\begin{document}


\begin{abstract}
Lupu introduced 
%in \cite{Lupu2014LoopsGFF}  % pas de citations dans l'abstract
a coupling between a random walk loop-soup and a Gaussian free field, where the sign of the field is constant on each cluster of loops. This coupling is a signed version of isomorphism theorems relating the square of the GFF to the occupation field of Markovian trajectories. 
His
%The 
construction 
%in \cite{Lupu2014LoopsGFF} 
starts with a loop-soup, and by adding additional randomness samples a GFF out of it.
In this article we provide the inverse construction: starting from a signed free field and using a self-interacting random walk related to this field, we construct a random walk loop-soup. Our construction relies on the previous work by Sabot and Tarrès, which inverts the coupling from the square of the GFF rather than the signed GFF itself.
%As a consequence, we also deduce an inversion of the coupling between the random current and the FK-Ising random cluster models 
%introduced by Lupu and Werner.
\end{abstract}
\maketitle
\section{Introduction}
\label{SecIntro}

The so called ``isomorphism theorems'' relate the square of a Gaussian free field (GFF) on an electrical network 
to occupation times of symmetric Markov jump processes
\cite{MarcusRosen2006MarkovGaussianLocTime,
Sznitman2012LectureIso}.
These date back to the work of Dynkin (Dynkin's isomorphism)
\cite{Dynkin1984IsomorphismPresentation,Dynkin1984Isomorphism, Dynkin1984PolynomOccupField}, and previously to Symanzik's
identities in Euclidean Quantum Field Theory \cite{Symanzik1969QFT} and Brydges-Fröhlich-Spencer random walk representation of spin systems \cite{BFS82Loop}.
Here we focus on the generalized second Ray-Knight theorem
\cite{ekmrs} and on Le Jan's isomorphism
\cite{LeJan2011Loops}.
The generalized second Ray-Knight theorem couples the squares of two GFFs with different, ordered, boundary conditions by adding the occupation times of independent Markovian excursions from boundary to boundary to the square with lower boundary conditions in order to obtain the square with higher boundary conditions. Le Jan's isomorphism states that the whole square of a GFF can be obtained as the occupation field of a Poisson Point Process of Markovian loops, known as loop-soup.
Le Jan's isomorphism in particular implies the 
generalized second Ray-Knight theorem.

In \cite{Lupu2014LoopsGFF} Lupu obtained ``signed'' or ``polarized'' versions
of isomorphism theorems, where one relates to the Markovian trajectories not only the square, but also the sign of the GFF. In particular the sign of the GFF is constant on each
Markovian trajectory. The construction goes through the introduction of the metric graph GFF. One first replaces each discrete edge of the electrical network by a continuous line, so as to obtain a continuous topological object, a one-dimensional simplicial complex known as metric graph or cable system, and then one interpolates the values of the GFF on the vertices by independent Brownian bridges inside the edges. This way one obtains a continuous Gaussian field. Its square can still be obtained as in Le Jan's isomorphism as an occupation field of a loop-soup of loops of the natural continuous diffusion on the metric graph. However, in this construction the sign components of the GFF are exactly the clusters of metric graph loops and the sign is chosen independently uniformly on each of them. 

Lupu's isomorphism has also a purely discrete description. One enlarges the clusters of the discrete loop-soup by opening the edges not visited by loops with certain probability, and then on each enlarged cluster one chooses a sign independently uniformly.

The above couplings have also a counterpart in the Ising ``world''. Indeed, conditional on the absolute value of the GFF, its sign is distributed like Ising spins with coupling constants given by the absolute value. Then, the enlarged
clusters in Lupu's isomorphism (discrete description),
conditional on the absolute value of the GFF on the vertices, are exactly Fortuin-Kasteleyn random clusters \cite{Grim2006FK}, with cluster weight $q=2$ and edge weights depending on the absolute value \cite{lupu-werner}.
FK random clusters with $q=2$ are coupled to the Ising spins through the Edwards-Sokal coupling \cite{EdwardsSokal88Ising}, where one simply chooses the spin independently uniformly on each clusters. The discrete loop-soup in Le Jan's isomorphism is related to the random current expansion of the Ising model \cite{lupu-werner,DC16RC,Lalwer2018Survey}.
Finally \cite{lupu-werner} connects the dots and showed that there is a natural coupling between Ising random currents and FK-Ising random clusters, which is actually Lupu's isomorphism conditioned on the absolute value of the GFF.
In Figure \ref{diagram} we summarize all above models and the couplings and relations between them.

\medskip

In this paper we deal with the inversion of Lupu's isomorphism, that is to say with retrieving the conditional law  of the discrete loop-soup given a discrete Gaussian free field (both its absolute value and its sign). 
This extends the work of Sabot and Tarrès who in
\cite{SabotTarres2015RK} gave the inversion of isomorphisms only in the case when the absolute value of the GFF was given,
or equivalently its square, but not its sign.
To fix the ideas, let us introduce some notations.

\medskip

Let $\mathcal{G}=(V,E)$ be a connected undirected graph, with $V$ at most countable and each vertex $x\in V$ of finite degree. We do not allow self-loops. Also in general we do not consider multiple edges, unless specified otherwise. 
Given $e\in E$ an edge, we will denote
$e_{+}$ and $e_{-}$ its end-vertices, even though $e$ is non-oriented and one can interchange $e_{+}$ and $e_{-}$. 
Each edge $e\in E$ is endowed with a conductance 
$W_{e}=W_{e_{-},e_{+}}=W_{e_{+},e_{-}}>0$. 
There may be a killing measure $\kappa=(\kappa_{x})_{x\in V}$ on vertices. 

We consider $(X_{t})_{t\ge0}$ the \textit{Markov jump processes} on $V$ which being in $x\in V$, jumps along an adjacent edge $e$ with rate $W_{e}$. Moreover if 
$\kappa_{x}\neq 0$, the process is killed at $x$ with rate $\kappa_{x}$ (the process is not defined after that time). $\zeta$ will denote the time up to which $X_{t}$ is defined. If $\zeta<+\infty$, then either the process has been killed by the killing measure $\kappa$ (and $\kappa \not\equiv 0$) or it has gone off to infinity in finite time 
(and $V$ infinite). We will assume that the process $X$ is transient, which means, if $V$ is finite, that 
$\kappa\not\equiv 0$. $\mathbb{P}_{x}$ will denote the law of $X$ started from $x$.
Let $(\textbf{G}(x,y))_{x,y\in V}$ be the Green function of 
$(X_{t})_{0\leq t<\zeta}$:
\begin{displaymath}
\textbf{G}(x,y)=\textbf{G}(y,x)=\mathbb{E}_{x}\left[\int_{0}^{\zeta} 
\indic_{\{X_{t}=y\}} dt\right].
\end{displaymath}
Let $\mathcal{E}$ be the Dirichlet form defined on functions $f$ on $V$ with finite support:
\begin{eqnarray}\label{Dirichlet-form}
\mathcal{E}(f,f)=\sum_{x\in V}\kappa_{x} f(x)^{2}+
\sum_{e\in E}W_e(f(e_{+})-f(e_{-}))^{2}.
\end{eqnarray}
$P_{\varphi}$ will be the law of $(\varphi_{x})_{x\in V}$ the centred 
\textit{Gaussian free field} (GFF) on $V$ with covariance
$E_{\varphi}[\varphi_{x}\varphi_{y}]=\textbf{G}(x,y)$. In case $V$ is finite, the density of $P_{\varphi}$ is
\begin{displaymath}
\dfrac{1}{(2\pi)^{\frac{\vert V\vert}{2}}\sqrt{\det \textbf{G}}}
\exp\left(-\dfrac{1}{2}\mathcal{E}(f,f)\right)\prod_{x\in V} df_{x}.
\end{displaymath}
$\varphi$ under $P_{\varphi}$ satisfies the 
\textit{spatial Markov property}. If $U$ is a subset of $V$ and
\begin{displaymath}
\partial U = \lbrace x\in U\vert\exists y\in V\setminus U, \text{ $x$ and $y$ joined by an edge $e\in E$}\rbrace,
\end{displaymath}
then $(\varphi_{x})_{x\in V\setminus U}$ conditional on
$(\varphi_{y})_{y\in U}$ has same law as conditional on
$(\varphi_{y})_{y\in \partial U}$.

Given $x_{0}\in V$ and $a\in\R$,
$P^{\lbrace x_{0}\rbrace,a}_{\varphi}$
will denote the law of
the GFF $\varphi$ conditioned to equal $a$ in $x_{0}$.
Note that if the killing measure $\kappa$ is supported in $x_{0}$, the law $P^{\lbrace x_{0}\rbrace,a}_{\varphi}$ does not depend on $\kappa$ and in this case one can as well take 
$\kappa= 0$.

We will denote by
$(\ell_{x}(t))_{x\in V, t\in [0,\zeta]}$ the family of local times of $X$:
\begin{displaymath}
\ell_{x}(t)=\int_{0}^{t}\indic_{\{X_{s}=x\}} ds.
\end{displaymath}
For all $x\in V$, $u>0$, let
\begin{displaymath}
\tau_{u}^{x}=\inf\lbrace t\geq 0; \ell_{x}(t)>u\rbrace.
\end{displaymath}
Recall the generalized second Ray-Knight theorem on discrete graphs by Eisenbaum, Kaspi, Marcus, Rosen and Shi \cite{ekmrs} (see also 
\cite{MarcusRosen2006MarkovGaussianLocTime,Sznitman2012LectureIso}):
\begin{theorem}[\textbf{Generalized second Ray-Knight theorem}]
\label{ThmRKGen}
For any $u>0$ and $x_{0}\in V$, 
\begin{center}
$\left(\ell_{x}(\tau_{u}^{x_{0}})+\dfrac{1}{2}\varphi_{x}^{2}\right)_{x\in V}$
under $\mathbb{P}_{x_{0}}(\cdot \vert \tau_{u}^{x_{0}}<\zeta)\otimes P^{\lbrace x_{0}\rbrace,0}_{\varphi}$
\end{center}
has the same law as 
\begin{center}
$\left(\dfrac{1}{2}\varphi_{x}^{2}\right)_{x\in V}$
under $P^{\lbrace x_{0}\rbrace,\sqrt{2u}}_{\varphi}$.
\end{center}
\end{theorem}

Sabot and Tarrès showed in \cite{SabotTarres2015RK} that the so-called ``magnetized'' reverse Vertex-Reinforced Jump Process (VRJP) provides an inversion of the generalized second Ray-Knight theorem, in the sense that it enables to retrieve the law of 
$(\ell_x(\tau_u^{x_0}), \varphi^2_x)_{x\in V}$ conditional on 
$\left(\ell_x(\tau_u^{x_0})+\frac{1}{2}\varphi^2_x\right)_{x\in V}$. The jump rates of that latter process 
are a product of a first factor accountable for a self-repulsion (reverse VRJP) and a second one which
can be interpreted as a ratio of two-point functions of the Ising model associated to time-evolving coupling constants
(magnetization). 
The process introduced in \cite{SabotTarres2015RK} 
also inverts the Le Jan's isomorphism
\cite{LeJan2011Loops} given the square of a GFF.

However in \cite{SabotTarres2015RK} the link with the Ising model is only implicit, and a natural question is whether Ray-Knight inversion can be described in a simpler form if we enlarge the state space of the dynamics, and in particular include the ``hidden'' spin variables. 

The answer is positive, and goes through a ``signed'' extension of the Ray-Knight isomorphism following Lupu's approach \cite{Lupu2014LoopsGFF}, which couples the sign of the GFF to the Markovian path. The Ray-Knight inversion will turn out to take a rather simple form  in Theorem \ref{thm-Poisson} of the present paper, where it will be defined not only through the spin variables but also random currents associated to the field though an extra Poisson Point Process. Further, in Theorem 
\ref{ThmPoissonLoopSoup} we will describe how the process we construct inverts Lupu's ``signed'' isomorphism between a loop-soup and a discrete GFF.

\medskip

The paper is organized as follows. 

In Section \ref{sec:srk} we recall some background on loop-soup isomorphisms and on related couplings and state and prove a signed version of the generalized second Ray-Knight theorem. We begin in Section \ref{sec:lejan} by a statement of Le Jan's isomorphism which couples the square of the Gaussian Free Field to the loop-soups, and recall how the generalized second Ray-Knight theorem can be seen as its corollary: for more details see \cite{lejan4}. In Section \ref{sec:lupu} we state Lupu's isomorphism which extends Le Jan's isomorphism and couples the  sign of the GFF to the loop-soups, using  a  metric graph extension of both the GFF and the Markov process. Lupu's isomorphism yields an interesting realization of the 
Edwards-Sokal  FK-Ising - spin Ising coupling,  and provides as well a ``Current+Bernoulli=FK'' coupling lemma \cite{lupu-werner}, which occur in the relationship between the discrete and metric graph versions. We briefly recall these couplings in Sections \ref{fkising} and \ref{randomcurrent}, as they are implicit in this paper. In Section \ref{sec:glupu} we state and prove the generalized second Ray-Knight ``version'' of Lupu's isomorphism, which we aim to invert. In Section
\ref{sec_diagram} we present a diagram which summarizes the models and the couplings.

Section \ref{sec:inversion} is devoted to the statements of inversions of these isomorphisms. We state in Section \ref{sec_Poisson} a signed version of the inversion of the generalized second Ray-Knight theorem through an extra Poisson Point Process, namely Theorem \ref{thm-Poisson}. In Section \ref{sec_dicr_time} we provide a discrete-time description of the process, whereas in Section \ref{sec_jump} we yield  an alternative version of that process through jump rates, which can be seen as an annealed version of the first one. 
The annealed process of Section \ref{sec_jump} is a reverse VRJP (self-repelling) which evolves on a subgraph of
$\mathcal{G}$ which itself shrinks over time. These subgraphs can be interpreted as FK-Ising random clusters associated to time-evolving, decreasing, edge weights.
In Section \ref{sec:lejaninv}
we deduce a signed inversion of Lupu's isomorphism for loop-soups.
%, and an inversion of the coupling of random current with FK-Ising in Section \ref{sec:coupinv}.%We retrieve in Section \ref{sec:back} the so-called ``magnetized'' reverse Vertex-Reinforced Jump Process as a version of %that inversion, annealed as well on the signs of the field. 

Finally Section \ref{sec:proof} is devoted to the proof of Theorem \ref{thm-Poisson}: Section \ref{sec:pfinite} deals with the case of a finite graph without killing measure, and Section \ref{sec:pgen} deduces the proof in the general case.  

\section{Le Jan's and Lupu's isomorphisms}
\label{sec:srk}
\subsection{Loop-soups and Le Jan's isomorphism}
\label{sec:lejan}
The \textit{loop measure} associated to the Markov jump process
$(X_{t})_{0\leq t<\zeta}$ is defined as follows. Let $\mathbb{P}^{t}_{x,y}$ be the bridge probability measure from 
$x$ to $y$ in time $t$ (conditional on $t<\zeta$). Let $p_{t}(x,y)$ be the transition probabilities of 
$(X_{t})_{0\leq t<\zeta}$. 

Let $\mu_{\rm loop}$ be the measure on time-parametrized nearest-neighbor based loops (i.e. loops with a starting site)
\begin{displaymath}
\mu_{\rm loop}=\sum_{x\in V}\int_{t>0}\mathbb{P}^{t}_{x,x} p_{t}(x,x) \dfrac{dt}{t}.
\end{displaymath}
The loops will be considered here up to a rotation of parametrisation (with the corresponding pushforward measure induced by $\mu_{\rm loop}$), that is to say a loop $(\gamma(t))_{0\leq t\leq t_{\gamma}}$ will be the same as 
$(\gamma(T+t))_{0\leq t\leq t_{\gamma}-T}\circ (\gamma(T+t-t_{\gamma}))_{t_{\gamma}-T\leq t\leq t_{\gamma}}$ for all
$T\in (0, t_{\gamma})$, 
where $\circ$ denotes the concatenation of paths.
A \textit{loop-soup} of intensity $\alpha>0$, denoted 
$\mathcal{L}_{\alpha}$, is a Poisson random measure of intensity
$\alpha \mu_{\rm loop}$. We see it as a random collection of loops in $\mathcal{G}$. Observe that a.s. above each vertex 
$x\in V$, $\mathcal{L}_{\alpha}$ contains infinitely many trivial ``loops'' reduced to the vertex $x$. There are also with positive probability non-trivial loop that visit several vertices. 

Let  $L_{.}(\mathcal{L}_{\alpha})$ be the \textit{occupation field} of $\mathcal{L}_{\alpha}$ on $V$ i.e., for all $x\in V$, 
\begin{displaymath}
L_x(\mathcal{L}_{\alpha})=
\sum_{(\gamma(t))_{0\leq t\leq t_{\gamma}}\in\mathcal{L}_{\alpha}}
\int_{0}^{t_{\gamma}}\indic_{\{\gamma(t)=x\}} dt.
\end{displaymath}
In \cite{LeJan2011Loops} Le Jan shows that for transient Markov jump processes, $L_x(\mathcal{L}_{\alpha})<+\infty$ for all $x\in V$ a.s. For $\alpha=1/2$ he identifies the law of
$L_.(\mathcal{L}_{\alpha})$:
\begin{theorem}[\textbf{Le Jan's isomorphism}]
\label{ThmIsoLeJan}
$L_.(\mathcal{L}_{1/2})=\left(L_x(\mathcal{L}_{1/2})\right)_{x\in V}$ has the same law as
\\$\dfrac{1}{2}\varphi^2=\left(\dfrac{1}{2}\varphi_{x}^{2}\right)_{x\in V}$
under $P_{\varphi}$.
\end{theorem}

Let us briefly recall how Le Jan's isomorphism enables one to retrieve the generalized second Ray-Knight theorem stated in Section \ref{SecIntro}: for more details, see for instance \cite{lejan4}. We assume that $\kappa$ is supported by $x_0$: the general case can be dealt with by an argument similar to the proof of Proposition \ref{PropKillingCase}. 
Let $D=V\setminus\{x_0\}$, and note that the isomorphism in particular implies that $L_.(\mathcal{L}_{1/2})$ conditional on $L_{x_0}(\mathcal{L}_{1/2})=u$ has the same law as $\varphi^2/2$ conditional on $\varphi_{x_0}^2/2=u$. 

On the one hand, given the classical energy decomposition, we have $\varphi=\varphi^D+\varphi_{x_0}$, with $\varphi^D$ the GFF associated to the restriction of $\mathcal{E}$ to $D$, where $\varphi^D$ and $\varphi_{x_0}$ are independent. Now $\varphi^2/2$ conditional on 
$\varphi_{x_0}^2/2=u$ has the law of $(\varphi^D+\eta\sqrt{2u})^2/2$, where $\eta$ is the sign of $\varphi_{x_0}$, which is independent of $\varphi^D$. But $\varphi^D$ is symmetric, so that the latter also has the law of $(\varphi^D+\sqrt{2u})^2/2$.

On the other hand, the loop-soup $\mathcal{L}_{1/2}$ can be decomposed into the two independent loop-soups $\mathcal{L}_{1/2}^D$ contained in $D$ and $\mathcal{L}_{1/2}^{(x_0)}$ hitting $x_0$. Now $L_.(\mathcal{L}_{1/2}^D)$ has the law of $(\varphi^D)^2/2$ and  $L_.(\mathcal{L}_{1/2}^{(x_0)})$ conditional on $L_{x_0}(\mathcal{L}_{1/2}^{(x_0)})=u$ has the law of the occupation field of the Markov chain 
$\ell(\tau_{u}^{x_{0}})$
under $\mathbb{P}_{x_{0}}(\cdot \vert \tau_{u}^{x_{0}}<\zeta)$, which enables us to conclude.

\subsection{Lupu's isomorphism}
\label{sec:lupu}
As in \cite{Lupu2014LoopsGFF}, we consider the \textit{metric graph} $\tilde{\mathcal{G}}$ 
(also known as \textit{cable system}) associated to 
$\mathcal{G}$. As a a topological space, it is obtained by replacing each discrete edge $e$ by a continuous compact line interval $I_{e}$. If two edges $e$ and $e'$ share a common extremity, the corresponding endpoints of $I_{e}$ and
$I_{e'}$ are identified. So $\tilde{\mathcal{G}}$ is a one-dimensional simplicial complex, with $0$-cells corresponding to vertices in $V$, and $1$-cells to edges in $E$. $V$ is naturally identified to a subset of $\tilde{\mathcal{G}}$. We further endow $\tilde{\mathcal{G}}$ with a metric by setting the length of each $I_{e}$ to be equal to 
$\frac{1}{2}W_{e}^{-1}$. We also consider the Radon measure
$\tilde m$ on $\tilde{\mathcal{G}}$, such that its restriction to each $I_{e}$ is a one-dimensional Lebesgue measure of total mass $\frac{1}{2}W_{e}^{-1}$ (i.e. the length measure).

One can define a standard Brownian motion 
$B^{\tilde{\mathcal{G}}}_{t}$ on 
$\tilde{\mathcal{G}}$ as follows. Inside each $I_{e}$,
$B^{\tilde{\mathcal{G}}}_{t}$ behaves as a standard one-dimensional Brownian motion. Upon reaching a vertex, 
$B^{\tilde{\mathcal{G}}}_{t}$ performs immediately independent Brownian excursions inside each adjacent edge, with equal rate for each direction, and that until eventually traversing one of the edge and reaching a neighbor vertex on the other side.
See Section 2 of \cite{Lupu2014LoopsGFF} for more details. Alternatively, 
$B^{\tilde{\mathcal{G}}}_{t}$ can be defined as the symmetric Markov process associated to the Dirichlet form
\begin{displaymath}
\tilde{\mathcal{E}}(f,f)=\dfrac{1}{2}
\int_{x\in\tilde{\mathcal{G}}}
f'(x)^{2}d\tilde m (x),
\end{displaymath}
where $f: \tilde{\mathcal{G}}\rightarrow \R$ is
continuous and $\mathcal{C}^{1}$ inside each of the
$I_{e}$. For more on Dirichlet forms and associated Markov processes, see \cite{FukushimaDirichletForms}.
$B^{\tilde{\mathcal{G}}}_{t}$ admits a family of Brownian local times $\tilde{\ell}_{x}(t)$, continuous in $(x,t)$, such that
for every bounded measurable function
$f: \tilde{\mathcal{G}}\rightarrow \R$ and every $t\geq 0$,
\begin{displaymath}
\int_{0}^{t} f(B^{\tilde{\mathcal{G}}}_{s}) ds=
\int_{x\in\tilde{\mathcal{G}}} f(x)\tilde{\ell}_{x}(t)
d\tilde{m}(x).
\end{displaymath}
Let be $\tilde{\zeta}$ the first time when either
$B^{\tilde{\mathcal{G}}}_{t}$ explodes to infinity
(if possible), or
\begin{displaymath}
\sum_{x\in V}\kappa_{x}\tilde{\ell}_{x}(t)
\end{displaymath}
hits an independent exponential r.v. of mean $1$
(if $\kappa\not\equiv 0$). If neither of these happens,
$\tilde{\zeta}=+\infty$.

One can deterministically recover the Markov jump process
$(X_{t})_{0\leq t<\zeta}$ out of
$(B^{\tilde{\mathcal{G}}}_{t})_{0\le t< \tilde{\zeta}}$
Let be
\begin{displaymath}
A_{V}(t)=\sum_{x\in V}\tilde{\ell}_{x}(t),
\qquad
A_{V}^{-1}(t)=\inf\lbrace s\geq 0\vert A_{V}(s)>t\rbrace.
\end{displaymath}
Then, if $B^{\tilde{\mathcal{G}}}_{0}\in V$,
$B^{\tilde{\mathcal{G}}}_{A_{V}^{-1}(t)}$ is a Markov jump process $X_{t}$ on $V$ with jump rates 
$(W_e)_{e\in E}$. 
Moreover, 
$A_{V}(\tilde{\zeta})=\zeta$ and 
$\ell_{x}(t)=\tilde{\ell}_{x}(A_{V}^{-1}(t))$ 
for $x\in V$. 

In \cite{Lupu2014LoopsGFF} Lupu introduces a measure 
$\tilde{\mu}_{\rm loop}$
on time-parametrized continuous loops on $\tilde{\mathcal{G}}$,
associated to the Brownian motion 
$(B^{\tilde{\mathcal{G}}}_{t})_{0\le t< \tilde{\zeta}}$.
$\tilde{\mathcal{L}}_{\alpha}$ will denote the Poisson Point Process of loops of intensity $\alpha\tilde{\mu}_{\rm loop}$. The discrete-space loops $\mathcal{L}_{\alpha}$ can be 
deterministically obtained from $\tilde{\mathcal{L}}_{\alpha}$ by taking the print of the latter on $V$, using the time-change
$A_{V}^{-1}$. Note that $\tilde{\mathcal{L}}_{\alpha}$ contains loops that do not visit $V$ and are entirely contained in one of the $I_{e}$. These do not contribute to 
$\mathcal{L}_{\alpha}$. $\tilde{\mathcal{L}}_{\alpha}$ has an occupation field
$(L_{x}(\tilde{\mathcal{L}}_{\alpha}))
_{x\in \tilde{\mathcal{G}}}$,
which is a sum over loops in $\tilde{\mathcal{L}}_{\alpha}$
of Brownian local times in $x$ of the loops. Moreover,
\begin{displaymath}
L_{x}(\tilde{\mathcal{L}}_{\alpha})=
L_{x}(\mathcal{L}_{\alpha}),~
\forall x\in V.
\end{displaymath}

Similarly, the GFF $\varphi$ on $\mathcal{G}$ with law 
$P_\varphi$ can be extended to a GFF $\tilde{\varphi}$ on 
$\tilde{\mathcal{G}}$ as follows. Given $e\in E$, one considers inside $I_{e}$ a conditionally independent Brownian bridge, actually a bridge of a $\sqrt{2} \times$ 
\textit{standard Brownian motion}, of length 
$\frac{1}{2}W_{e}^{-1}$, with end-values
$\varphi_{e_{-}}$ and $\varphi_{e_{+}}$. This provides a continuous Gaussian field on the metric graph which still satisfies the spatial Markov property.

Lupu introduced in \cite{Lupu2014LoopsGFF} an isomorphism linking  the GFF $\tilde{\varphi}$ and the loop-soup 
$\tilde{\mathcal{L}}_{1/2}$ on the metric graph
$\tilde{\mathcal{G}}$. This one uses the clusters of
$\tilde{\mathcal{L}}_{1/2}$. First of all, \textit{cluster}
is an equivalence class of loops in 
$\tilde{\mathcal{L}}_{1/2}$, where two loops
$\tilde \gamma$ and $\tilde \gamma'$ belong to the same cluster if there is a finite chain 
$\tilde \gamma_{0},\dots, \tilde \gamma_{k}$
of loops in $\tilde{\mathcal{L}}_{1/2}$ such that
$\tilde \gamma_{0}=\gamma$,
$\tilde \gamma_{k}=\gamma'$, and for
$i\in\lbrace 1,\dots, k\rbrace$ the loops
$\tilde \gamma_{i-1}$ and $\tilde \gamma_{i}$
(their ranges) have a non-empty intersection.
By extension, a cluster will also be a (connected) subset of
$\tilde{\mathcal{G}}$ obtained as a union 
over a cluster of loops of ranges of the loops.
Note that if $\mathcal{L}_{1/2}$ is obtained as a print of
$\tilde{\mathcal{L}}_{1/2}$ on $V$, each cluster of
$\mathcal{L}_{1/2}$ is contained in a cluster of
$\tilde{\mathcal{L}}_{1/2}$, but in general a cluster
$\tilde{\mathcal{L}}_{1/2}$ may correspond to several different clusters of $\mathcal{L}_{1/2}$ merged together. This is because there are more connections at the level of the metric graph.

\begin{theorem}[\textbf{Lupu's isomorphism}]
\label{thm:Lupu} 
There is a coupling between the Poisson ensemble of loops 
$\tilde{\mathcal{L}}_{1/2}$ and 
$(\tilde{\varphi}_y)_{y\in\tilde{\mathcal{G}}}$ 
defined above, such that the two following constraints hold: 
\begin{itemize}
\item For all $y\in\tilde{\mathcal{G}}$, $L_y(\tilde{{\mathcal{L}}}_{1/2})=\frac{1}{2}\tilde{\varphi}_y^2$;
\item the zero set of 
$\tilde\varphi$,
$\lbrace x\in \tilde{\mathcal{G}}\vert 
\tilde\varphi_{x}=0\rbrace$, is the set of points in 
$\tilde{\mathcal{G}}$ not visited by any loop of
$\tilde{\mathcal{L}}_{1/2}$;
\item the clusters of loops of $\tilde{\mathcal{L}}_{1/2}$ are exactly the sign clusters of $(\tilde{\varphi}_y)_{y\in\tilde{\mathcal{G}}}$.
\end{itemize}
Conditional on $(|\tilde{\varphi}_y|)_{y\in\tilde{\mathcal{G}}}$, the sign of $\tilde{\varphi}$ on each of its connected components is distributed independently and uniformly in 
$\{-1,+1\}$. 
\end{theorem}
Lupu's isomorphism and the idea of using metric graphs were applied in \cite{Lupu2015ConvCLE} to show that on the discrete half-plane $\mathbb{Z}\times\mathbb{N}$, the scaling limits of outermost boundaries of clusters of loops in loop-soups are the Conformal Loop Ensembles $\hbox{CLE}$.

Let $\mathcal{O}(\tilde{\varphi})$ (resp. $\mathcal{O}(\tilde{\mathcal{L}}_{1/2})$) be the set of edges $e\in E$ such that $\tilde{\varphi}$ (resp. $\tilde{\mathcal{L}}_{1/2}$) does not touch $0$ on $I_{e}$, in other words such that all the edge-interval $I_{e}$ remains in the same sign cluster of 
$\tilde{\varphi}$ (resp. loop cluster of
$\tilde{\mathcal{L}}_{1/2}$). 
$\mathcal{O}(\tilde{\varphi})$ and 
$\mathcal{O}(\tilde{\mathcal{L}}_{1/2})$ have same law.
Let 
$\mathcal{O}(\mathcal{L}_{1/2})$ be the set of edges $e\in E$ that are crossed (i.e. endpoints visited consecutively) by 
loops in $\mathcal{L}_{1/2}$ (obtained as the print on $V$ of
$\tilde{\mathcal{L}}_{1/2}$).

In order to translate Lupu's isomorphism back onto the initial graph $\mathcal{G}$, one needs to describe on one hand the distribution of $\mathcal{O}(\tilde{\varphi})$ conditional on the values of $(\varphi_{x})_{x\in V}$, and on the other hand the distribution of $\mathcal{O}(\tilde{\mathcal{L}}_{1/2})$  conditional on $\mathcal{L}_{1/2}$ and the cluster of loops $\mathcal{O}(\mathcal{L}_{1/2})$ on the discrete graph $\mathcal G$. These two distributions are described respectively in Sections \ref{fkising} and \ref{randomcurrent}, and provide realisations of the  Edwards-Sokal FK-Ising - spin Ising coupling and of the ``Current+Bernoulli=FK'' coupling lemma \cite{lupu-werner}.

\subsection{The FK-Ising distribution of $\mathcal{O}(\tilde{\varphi})$ conditional on $|\varphi|$}
\label{fkising}
\begin{lemma}
\label{lem:fki}
Conditional on $(\varphi_{x})_{x\in V}$, 
$(\indic_{\{e\in \mathcal{O}(\tilde{\varphi})\}})_{e\in E}$
is a family of independent random variables and
\begin{displaymath}
\mathbb{P}\left(e\not\in \mathcal{O}(\tilde{\varphi})\vert \varphi\right)=
\left\lbrace
\begin{array}{ll}
1 & \text{if}~ \varphi_{e_{-}}\varphi_{e_{+}}<0,\\ 
\exp\left(-2W_{e}\varphi_{e_{-}}\varphi_{e_{+}}\right)
& \text{if}~ \varphi_{e_{-}}\varphi_{e_{+}}>0.
\end{array} 
\right.
\end{displaymath}
\end{lemma}
\begin{proof}
Conditional on $(\varphi_{x})_{x\in V}$, 
the metric graph GFF
$(\tilde{\varphi}_y)_{y\in\tilde{\mathcal{G}}}$ 
is constructed by adding independent Brownian bridges on each edge, so that $(\indic_{\{e\in \mathcal{O}(\tilde{\varphi})\}})_{e\in E}$ are conditionally independent random variables, and it follows from the reflection principle for one-dimensional Brownian motion that, if $\varphi_{e_{-}}\varphi_{e_{+}}>0$, then 
\begin{displaymath}
\mathbb{P}\left(e\not\in \mathcal{O}(\tilde{\varphi})\vert \varphi,\varphi_{e_{-}}\varphi_{e_{+}}>0\right)=
\dfrac{\exp\left(-\frac{1}{2}W_{e}(\varphi_{e_{-}}+\varphi_{e_{+}})^{2}\right)}
{\exp\left(-\frac{1}{2}W_{e}(\varphi_{e_{-}}-\varphi_{e_{+}})^{2}\right)}=
\exp\left(-2W_{e}\varphi_{e_{-}}\varphi_{e_{+}}\right).
\qedhere
\end{displaymath}
\end{proof}

Let us now recall how the conditional probability in Lemma \ref{lem:fki} yields a realization of the FK-Ising coupling. 

 Assume $V$ is finite. Let $(J_{e})_{e\in E}$ be a family of positive weights. A \textit{(spin) Ising model} on $V$ with interaction constants $(J_{e})_{e\in E}$ is a
probability on configuration of spins $({\sigma}_{x})_{x\in V}\in \{+1,-1\}^V$ such that
\begin{displaymath}
\mathbb{P}^{\rm Isg}_{J}((\sigma_x)_{x\in V})=
\dfrac{1}{\mathcal{Z}^{\rm Isg}_{J}}\exp\left(\sum_{e\in E}
J_{e}\sigma_{e_{-}}\sigma_{e_{+}}\right).
\end{displaymath}
An \textit{FK-Ising random cluster model}
\cite{Grim2006FK} with edge weights
$(1-e^{-2J_{e}})_{e\in E}$ is a random configuration of open (value $1$) and closed
edges (value $0$) such that
\begin{displaymath}
\mathbb{P}^{\rm FK-Isg}_{J}((\omega_{e})_{e\in E})=
\dfrac{1}{\mathcal{Z}^{\rm FK-Isg}_{J}}
2^{\#\text{clusters}}
\prod_{e\in E}(1-e^{-2J_{e}})^{\omega_{e}}(e^{-2J_{e}})^{1-\omega_{e}},
\end{displaymath}
where "$\#\text{clusters}$" denotes the number of clusters created by open edges.

The Edwards-Sokal \cite{EdwardsSokal88Ising} coupling between FK-Ising and spin Ising reads as follows:
\begin{theorem}[\textbf{Edwards-Sokal coupling}]
\label{FK-Ising}
Given an FK-Ising model, sample on each cluster an independent uniformly distributed spin. The spins are then distributed according to the Ising model. Conversely, given a spin configuration
$\hat{\sigma}$ following the Ising distribution, consider each edge $e$, such that
$\hat{\sigma}_{e_{-}}\hat{\sigma}_{e_{+}}<0$, closed, and each edge $e$, such that
$\hat{\sigma}_{e_{-}}\hat{\sigma}_{e_{+}}>0$ open with probability
$1-e^{-2J_{e}}$. Then the open edges are distributed according to the FK-Ising model.
The two couplings between FK-Ising and spin Ising are the same.
\end{theorem}

Consider the GFF $\varphi$ on $\mathcal{G}$ distributed according to $P_{\varphi}$. Let
$J_{e}(\vert\varphi\vert)$ be the random interaction constants
\begin{displaymath}
J_{e}(\vert\varphi\vert)=W_{e}\vert\varphi_{e_{-}}\varphi_{e_{+}}\vert.
\end{displaymath}

Conditional on $\vert\varphi\vert$,
$(\operatorname{sign}(\varphi_{x}))_{x\in V}$ follows an Ising distribution with interaction constants $(J_{e}(\vert\varphi\vert))_{e\in E}$:
indeed, the Dirichlet form \eqref{Dirichlet-form} can be written as
\begin{equation}
\label{EqDirichlet}
\mathcal{E}(\varphi,\varphi)=\sum_{x\in V}\kappa_{x} \varphi_{x}^{2}+
\sum_{x\in V}W_{x}\varphi_{x}^{2}-
2\sum_{e\in E}J_e(\vert\varphi\vert) \operatorname{sign}(
\varphi_{e_{+}})\operatorname{sign}(\varphi_{e_{-}}),
\end{equation}
where
$$W_{x}=\sum_{\substack{y\in V\\y\sim x}} 
W_{x,y},$$
$y\sim x$ meaning that $x$ and $y$ are joined by an edge.
Similarly, when $\varphi$ distributed according to 
$P_{\varphi}^{\{x_0\},\sqrt{2u}}$ has boundary condition $\sqrt{2u}\ge 0$ on $x_0$, then $(\operatorname{sign}(\varphi_{x}))_{x\in V}$ has an Ising distribution
with interaction $(J_{e}(\vert\varphi\vert))_{e\in E}$ and conditioned on $\sigma_{x_0}=+1$.

Now, conditional on $\varphi$, $\mathcal{O}(\tilde{\varphi})$ has FK-Ising distribution  with weights
$(1-e^{-2J_{e}(\vert\varphi\vert)})_{e\in E}$. Indeed, the probability for $e\in\mathcal{O}(\tilde{\varphi})$ conditional on $\varphi$ is $1-e^{-2J_{e}(\vert\varphi\vert)}$,  by Lemma \ref{lem:fki}, as in Theorem \ref{FK-Ising}. 

Note that, given that $\mathcal{O}(\tilde{\varphi})$ has FK-Ising distribution, the fact that the sign on its connected components is distributed independently and uniformly in $\{-1,1\}$ can be seen either as a consequence of Theorem \ref{FK-Ising}, or from Theorem \ref{thm:Lupu}.

Given $\varphi=(\varphi_x)_{x\in V}$ on the discrete graph $\mathcal{G}$, we introduce in Definition \ref{def_FK-Ising} the random set of edges which has the distribution of $\mathcal{O}(\tilde{\varphi})$ conditional on $\varphi=(\varphi_x)_{x\in V}$. 
\begin{definition}\label{def_FK-Ising}
We let $\mathcal{O}(\varphi)$ be a random set of edges which has the distribution of $\mathcal{O}(\tilde{\varphi})$ conditional on 
$\varphi=(\varphi_x)_{x\in V}$ given by Lemma \ref{lem:fki}.
\end{definition}
\subsection{Distribution of $\mathcal{O}(\tilde{\mathcal{L}}_{1/2})$  conditional on $\mathcal{L}_{1/2}$ }
\label{randomcurrent}
The distribution of $\mathcal{O}(\tilde{\mathcal{L}}_{1/2})$  conditional on $\mathcal{L}_{1/2}$ can be retrieved by  Corollary 3.6 in \cite{Lupu2014LoopsGFF}, which reads as follows. 
\begin{lemma}[Corollary 3.6 in \cite{Lupu2014LoopsGFF}]
\label{36} Conditional on $\mathcal{L}_{1/2}$, the events $e\not\in\mathcal{O}(\tilde{\mathcal{L}}_{1/2})$, $e\in E\setminus\mathcal{O}(\mathcal{L}_{1/2})$, are independent and have probability
\begin{equation}
\label{cp}
\exp\left(-2 W_{e} \sqrt{L_{e_{+}}(\mathcal{L}_{1/2})L_{e_{-}}(\mathcal{L}_{1/2})}\right).
\end{equation}
\end{lemma}
This result gives rise, together with Theorem \ref{thm:Lupu}, to the following discrete version of Lupu's isomorphism, which is stated without any recourse to the metric graph induced by $\mathcal{G}$. 
\begin{definition}
\label{def:out}
Let $(\omega_{e})_{e\in E}\in\lbrace 0,1\rbrace^{E}$ be a percolation defined as follows: conditional on $\mathcal{L}_{1/2}$, the random variables 
$(\omega_{e})_{e\in E}$ are independent, and $\omega_{e}$ equals $0$ with conditional probability given by \eqref{cp}.

Let $\mathcal{O}_{+}(\mathcal{L}_{1/2})$ be the set of edges:
\begin{displaymath}
\mathcal{O}_{+}(\mathcal{L}_{1/2})=\mathcal{O}(\mathcal{L}_{1/2})
\cup \lbrace e\in E\vert \omega_{e}=1\rbrace.
\end{displaymath}
\end{definition}

\begin{theorem}[\textbf{Discrete version of Lupu's isomorphism}, Theorem 1 bis in \cite{Lupu2014LoopsGFF}]
\label{PropIsoLupuLoops}
Given a loop-soup $\mathcal{L}_{1/2}$, let $\mathcal{O}_{+}(\mathcal{L}_{1/2})$ be as in Definition \ref{def:out}.
Let $(\sigma_{x})_{x\in V}\in\lbrace -1,+1\rbrace^{V}$ be random spins taking constant values on
clusters induced by $\mathcal{O}_{+}(\mathcal{L}_{1/2})$ 
($\sigma_{e_{-}}=\sigma_{e_{+}}$ if $e\in \mathcal{O}_{+}(\mathcal{L}_{1/2})$) and such that the values on each cluster, conditional on $\mathcal{L}_{1/2}$ and $\mathcal{O}_{+}(\mathcal{L}_{1/2})$, are independent and uniformly distributed. Then
\begin{displaymath}
\left(\sigma_{x}\sqrt{2 L_{x}(\mathcal{L}_{1/2})}\right)_{x\in V}
\end{displaymath}
is a Gaussian free field distributed according to $P_{\varphi}$.
\end{theorem}

Theorem \ref{PropIsoLupuLoops}  induces the following coupling between FK-Ising and random currents. 

If $V$ is finite, a \textit{random current model} 
\cite{Aizenman82RCphi4,DC16RC}
on $\mathcal{G}$ with weights
$(J_{e})_{e\in E}$ is a random assignment to each edge $e$ of a non-negative integer
$\hat{n}_{e}$ such that for all $x\in V$,
\begin{displaymath}
\sum_{\substack{e\in E\\e~\text{adjacent to}~x}}\hat{n}_{e}
\end{displaymath}
is even, which is called the \textit{parity condition}. The probability of a configuration
$(n_{e})_{e\in E}$ satisfying the parity condition is
\begin{displaymath}
\mathbb{P}^{\rm RC}_{J}(\forall e\in E, \hat{n}_{e}=n_{e})=
\dfrac{1}{\mathcal{Z}^{\rm RC}_{J}}\prod_{e\in E}\dfrac{(J_{e})^{n_{e}}}{n_{e}!},
\end{displaymath}
where actually $\mathcal{Z}^{\rm RC}_{J}=\mathcal{Z}^{\rm Isg}_{J}$. Let
\begin{displaymath}
\mathcal{O}(\hat{n})=\lbrace e\in E\vert \hat{n}_{e}>0\rbrace.
\end{displaymath}
The open edges in $\mathcal{O}(\hat{n})$ induce clusters on the graph $\mathcal{G}$. 

Given a loop-soup $\mathcal{L}_{\alpha}$, we denote by 
$\mathsf{N}_{e}(\mathcal{L}_{\alpha})$ the number of times the loops in $\mathcal{L}_{\alpha}$ cross the nonoriented edge $e\in E$. The transience of the Markov jump process $X$ implies that
$\mathsf{N}_{e}(\mathcal{L}_{\alpha})$ is a.s. finite for all $e\in E$. If $\alpha=1/2$, we have the following identity (see for instance \cite{Werner2015,Lalwer2018Survey}):

\begin{theorem}[\textbf{Loop-soup and random current}]
\label{ThmLSRC}
Assume $V$ is finite and consider the loop-soup $\mathcal{L}_{1/2}$. Conditional on the occupation field $(L_{x}(\mathcal{L}_{1/2}))_{x\in V}$, 
$(\mathsf{N}_{e}(\mathcal{L}_{1/2}))_{e\in E}$ is distributed as a random current with weights 
$\left(2W_{e}\sqrt{L_{e_{-}}(\mathcal{L}_{1/2})L_{e_{+}}
(\mathcal{L}_{1/2})}\right)_{e\in E}$. If $\varphi$ is the GFF on $\mathcal{G}$ given by Le Jan's or Lupu's isomorphism, then these weights are
$(J_{e}(\vert\varphi\vert))_{e\in E}$.
\end{theorem}

Conditional on the occupation field 
$(L_{x}(\mathcal{L}_{1/2}))_{x\in V}$, 
$\mathcal{O}(\mathcal{L}_{1/2})$ are the edges occupied by a random current and 
$\mathcal{O}_{+}(\mathcal{L}_{1/2})$ the edges occupied by FK-Ising. 
Lemma \ref{lem:fki} and Theorem \ref{PropIsoLupuLoops} imply the following coupling, as noted by Lupu and Werner in 
\cite{lupu-werner}.

\begin{theorem}[\textbf{Random current and FK-Ising coupling}]
\label{RCFKIsing}
Assume $V$ is finite. Let $\hat{n}$ be a random current on $\mathcal{G}$ with weights
$(J_{e})_{e\in E}$. Let $(\omega_{e})_{e\in E}\in\lbrace 0,1\rbrace^{E}$ be an independent percolation, each edge being opened (value $1$) independently with probability 
$1-e^{-J_{e}}$. Then
\begin{displaymath}
\mathcal{O}(\hat{n})\cup\lbrace e\in E\vert \omega_{e}=1\rbrace
\end{displaymath}
is distributed like the open edges in an FK-Ising with weights 
$(1-e^{-2 J_{e}})_{e\in E}$.
\end{theorem}
 
%\end{IsgFKIsg}
\subsection{Generalized second Ray-Knight ``version'' of Lupu's isomorphism}\label{sec:glupu}
We are now in a position to state the coupled version of the second Ray-Knight theorem.
\begin{theorem}
\label{Lupu}
%Assume that the set of vertices $V$ is finite, and 
%La condition de finitude est ici inutile
Let $x_{0}\in V$. Let be $(\varphi_{x}^{(0)})_{x\in V}$ with distribution $P_{\varphi}^{\lbrace x_{0}\rbrace,0}$, and define $\mathcal{O}(\varphi^{(0)})$ as in Definition \ref{def_FK-Ising}. 
Let $X$ be an independent Markov jump process started from $x_{0}$. 

Fix $u>0$. If $\tau_{u}^{x_{0}}<\zeta$, we 
let $\mathcal{O}_{u}$ be the random subset of $E$ which contains $\mathcal{O}(\varphi^{(0)})$, the edges used by the path $(X_{t})_{0\leq t\leq \tau_{u}^{x_{0}}}$, and additional edges $e$ opened conditionally independently with probability
\begin{displaymath}
1-e^{W_{e}\vert\varphi_{e_{-}}^{(0)}\varphi_{e_{+}}^{(0)}\vert - 
W_{e}\sqrt{(\varphi_{e_{-}}^{(0)2}+2\ell_{e_{-}}(\tau_{u}^{x_{0}}))
(\varphi_{e_{+}}^{(0)2}+2\ell_{e_{+}}(\tau_{u}^{x_{0}}))}}.
\end{displaymath}
We let $\sigma\in\lbrace -1,+1\rbrace^{V}$ be random spins sampled uniformly independently on each cluster induced by
$\mathcal{O}_{u}$, pinned at $x_0$, i.e. $\sigma_{x_0}=1$, and define
\begin{displaymath}
\varphi_{x}^{(u)}:=\sigma_{x}\sqrt{\varphi_{x}^{(0)2}+2\ell_{x}(\tau_{u}^{x_{0}})}.
\end{displaymath}

Then, conditional on $\tau_{u}^{x_{0}}<\zeta$, $\varphi^{(u)}$ has  distribution $P_{\varphi}^{\lbrace x_{0}\rbrace,\sqrt{2u}}$, and
$\mathcal{O}_{u}$ has distribution $\mathcal{O}(\varphi^{(u)})$ conditional on $\varphi^{(u)}$.
\end{theorem}
\begin{remark}
One consequence of that coupling is that the path $(X_{s})_{s\le \tau_{u}^{x_{0}}}$ stays in the 
positive connected component of ${x_0}$ for $\varphi^{(u)}$. This yields a coupling between the range
of the Markov chain and the sign component of $x_{0}$ inside a GFF $P_{\varphi}^{\lbrace x_{0}\rbrace,\sqrt{2u}}$.
\end{remark}

\noindent{\it Proof of Theorem \ref{Lupu}:~}
The proof is based on
\cite{Lupu2014LoopsGFF}.  Let $D=V\setminus\{x_0\}$, and let $\tilde{\mathcal{L}}_{1/2}$ be the loop-soup of intensity $1/2$ on the metric graph $\tilde{\mathcal{G}}$, which we decompose into $\tilde{\mathcal{L}}_{1/2}^{(x_0)}$ (resp. $\tilde{\mathcal{L}}_{1/2}^{D}$) the loop-soup hitting (resp. not hitting) $x_0$, which are independent. We let $\mathcal{L}_{1/2}$ and $\mathcal{L}_{1/2}^{(x_0)}$ (resp.  $\mathcal{L}_{1/2}^{D}$) be the prints of these loop-soups on $V$ (resp. on $D=V\setminus\{x_0\}$). We condition on $L_{x_0}(\mathcal{L}_{1/2})=u$. 

Theorem \ref{thm:Lupu} implies (recall also Definition \ref{def_FK-Ising}) that we can couple $\tilde{\mathcal{L}}_{1/2}^{D}$ with $\varphi^{(0)}$ so that 
$L_x(\mathcal{L}_{1/2}^{D})=\varphi_x^{(0)2}/2$ for all $x\in V$, and 
$\mathcal{O}(\tilde{\mathcal{L}}_{1/2})=\mathcal{O}(\varphi^{(0)})$. 

Define 
$\varphi^{(u)}=(\varphi^{(u)}_x)_{x\in V}$ from 
$\tilde{\mathcal{L}}_{1/2}$ by, for all $x\in V$, 
\begin{equation*}
\label{abs}
|\varphi_x^{(u)}|=\sqrt{2L_x(\mathcal{L}_{1/2})}
\end{equation*}
and $\varphi_x^{(u)}=\sigma_x|\varphi_x^{(u)}|$, where $\sigma\in\{-1,+1\}^V$ are random spins sampled uniformly independently on each cluster induced by $\mathcal{O}(\tilde{\mathcal{L}}_{1/2})$, pinned at $x_0$,  i.e. $\sigma_{x_0}=1$. Then, by Theorem \ref{thm:Lupu}, 
$\varphi^{(u)}$ has distribution $P_{\varphi}^{\lbrace x_{0}\rbrace,\sqrt{2u}}$.

For all $x\in V$, we have
$$L_x(\tilde{\mathcal{L}}_{1/2})=\dfrac{1}{2}\varphi_x^{(0)2}+L_x(\mathcal{L}_{1/2}^{(x_0)}).$$

On the other hand, conditional on $L_.(\mathcal{L}_{1/2})$,
\begin{align*}
&\mathbb{P}(e\not\in\mathcal{O}(\tilde{\mathcal{L}}_{1/2})\,|\,
e\not\in\mathcal{O}(\tilde{\mathcal{L}}_{1/2}^D)\cup\mathcal{O}(\mathcal{L}_{1/2}))
=\frac{\mathbb{P}(e\not\in\mathcal{O}(\tilde{\mathcal{L}}_{1/2}))}{\mathbb{P}(e\not\in\mathcal{O}(\tilde{\mathcal{L}}_{1/2}^D)\cup\mathcal{O}(\mathcal{L}_{1/2}))}
\\&=
\frac{\mathbb{P}(e\not\in\mathcal{O}(\tilde{\mathcal{L}}_{1/2})\,|\,
e\not\in\mathcal{O}(\mathcal{L}_{1/2}))}{\mathbb{P}(e\not\in\mathcal{O}(\tilde{\mathcal{L}}_{1/2}^D)\,|\,
e\not\in\mathcal{O}(\mathcal{L}_{1/2}))}
=\frac{\mathbb{P}(e\not\in\mathcal{O}(\tilde{\mathcal{L}}_{1/2})\,|\,
e\not\in\mathcal{O}(\mathcal{L}_{1/2}))}{\mathbb{P}(e\not\in\mathcal{O}(\tilde{\mathcal{L}}_{1/2}^D)\,|\,
e\not\in\mathcal{O}(\mathcal{L}_{1/2}^D))}\\&
=\exp\left(-W_e\sqrt{L_{e_-}(\mathcal{L}_{1/2})L_{e_+}(\mathcal{L}_{1/2})}
+W_e\sqrt{L_{e_-}(\mathcal{L}_{1/2}^D)L_{e_+}(\mathcal{L}_{1/2}^D)}\right),
\end{align*}
where we use in the third equality that the event $e\not\in\mathcal{O}(\tilde{\mathcal{L}}_{1/2}^D)$ is measurable with respect to the $\sigma$-field generated by $\tilde{\mathcal{L}}_{1/2}^D$, which is independent of $\tilde{\mathcal{L}}_{1/2}^{(x_0)}$, and where we use Lemma \ref{36} in the fourth equality, for $\tilde{\mathcal{L}}_{1/2}$ and for $\tilde{\mathcal{L}}_{1/2}^D$. 

We conclude the proof by observing that  $\mathcal{L}_{1/2}^{(x_0)}$ conditional on $L_{x_0}(\mathcal{L}_{1/2}^{(x_0)})=u$ has the law of the occupation field of the Markov chain $\ell(\tau_{u}^{x_{0}})$
under $\mathbb{P}_{x_{0}}(\cdot \vert \tau_{u}^{x_{0}}<\zeta)$.
{\hfill $\Box$}

\newpage

\subsection{A diagram to summarize the models and the couplings and relations}
\label{sec_diagram}
Next diagram summarizes the preceding:

\begin{figure}[H]
\centering
\caption{Couplings between discrete and metric graph GFFs, loop-soups, and their relations to the spin Ising, the FK-Ising and the Ising random current. 
The first column corresponds to metric graph GFF, the second to discrete GFF and the third to Ising. 
Straightforward links are not labeled.}

%\bigskip

\begin{tikzpicture}
  \matrix (m)
    [
      matrix of nodes,
      %nodes={align=left, text width=2.75cm},
      column sep      = 5em,
      row sep         = 10ex,
      column 1/.style = { nodes = 
      { state, align=center, text width=2.75cm}  },
      column 2/.style = { nodes = 
      { state, align=center, text width=2.75cm} },
      column 3/.style = { nodes = 
      { state, align=center, text width=2.7cm} },
    ]
    { {squared metric \\ GFF $(\tilde{\varphi}_{x}^{2})
    	_{x\in\tilde{\mathcal{G}}}$}
     & {squared discrete \\ GFF $(\varphi_{x}^{2})_{x\in V}$}
     & \\
      {metric graph \\ 
      GFF $(\tilde{\varphi}_{x})_{x\in\tilde{\mathcal{G}}}$}  
      & {discrete \\ GFF $(\varphi_{x})_{x\in V}$} 
      & {Ising \\ spins $(\hat{\sigma}_{x})_{x\in V}$} \\
      {sign\\ clusters \\ 
      $\ooo(\tilde{\varphi})=\ooo(\tilde{\mathcal{L}}_{1/2})$} 
      & {enlarged \\ clusters \\$\ooo_{+}(\mathcal{L}_{1/2})$}  
      & {FK-Ising \\ random cluster \\ model} \\
      & {discrete loop \\ clusters 
      $\ooo(\mathcal{L}_{1/2})$}
      & {random current \\ clusters $\ooo(\hat{n})$} \\
      & {number of \\ edge crossings 
      $(\mathsf{N}_{e}(\mathcal{L}_{1/2}))_{e\in E}$} 
      & {Ising \\ random \\ current $(\hat{n}_{e})_{e\in E}$} 
      \\
      {metric graph \\ loop-soup $\tilde{\mathcal{L}}_{1/2}$} 
      & {discrete \\ loop-soup $\mathcal{L}_{1/2}$}
      &  \\
    };
    \draw [arrow] (m-2-1) -- (m-1-1); 
    \draw [arrow] ([xshift=-1.4cm]m-2-1.south east) -- ([xshift=-1.4cm]m-3-1.north east);
    \draw [arrow] ([xshift=1.4cm]m-3-1.north west) --
    node[anchor=east,left, align=center]
    {i.i.d. \\ unif. \\ signs} 
    ([xshift=1.4cm]m-2-1.south west);
    \draw [arrow] (m-6-1) -- 
    node[anchor=west,right]{clusters} (m-3-1);
    \draw [arrow] (m-2-2) -- (m-1-2);
    \draw [arrow] ([xshift=-1.4cm]m-2-2.south east) --
   node[anchor=west,right,xshift=-0.2em,yshift=-0.3em]
   {Lem. \ref{lem:fki}} ([xshift=-1.4cm]m-3-2.north east);
    \draw [arrow] ([xshift=1.4cm]m-3-2.north west) -- 
    node[anchor=east,left, align=center]
    {i.i.d. \\ unif. \\ signs} 
    ([xshift=1.4cm]m-2-2.south west);
    \draw [arrow] (m-4-2) -- 
     node[anchor=east,left]{Lem. \ref{36}} (m-3-2);
    \draw [arrow] (m-5-2) -- node[anchor=east,left,align=right]
    {$\mathsf{N}_{e}(\mathcal{L}_{1/2})$ \\ $>0$} (m-4-2);
    \draw [arrow] (m-6-2) -- (m-5-2);
    \draw [darrow] (m-2-3) -- 
    node[anchor=west,right]{Thm. \ref{FK-Ising}} (m-3-3);
    \draw [arrow] (m-4-3) --  
    node[anchor=west,right]{Thm. \ref{RCFKIsing}} (m-3-3);
    \draw [arrow] (m-5-3) -- 
    node[anchor=west,right]{$\hat{n}_{e}>0$} (m-4-3);
    \draw [arrow] (m-1-1) -- node[anchor=north,below, align=center]{restr. \\ to $V$} (m-1-2);
    \draw [arrow] (m-2-1) -- node[anchor=north,below, align=center]{restr. \\ to $V$} (m-2-2);
    \draw [darrow] (m-3-1) -- 
    node[anchor=north,below,xshift=-0.2em]
    {equal} (m-3-2);
    \draw [arrow] (m-6-1) -- node[anchor=north,below, align=center]{trace \\ on $\mathcal{G}$} (m-6-2);
    \draw [arrow] (m-2-2) -- 
    node[anchor=north,below, align=center,xshift=0.3em]
    {sign$(\varphi)$ \\ cond. \\ on $\vert\varphi\vert$} (m-2-3);
    \draw [arrow] (m-3-2) -- 
    node[anchor=north,below, align=center,xshift=0.3em]
    {cond. \\ on $\vert\varphi\vert$} (m-3-3);
    \draw [arrow] (m-4-2) -- node[anchor=north,below,xshift=0.3em,yshift=-0.1em]
    {Thm. \ref{ThmLSRC}} (m-4-3);
    \draw [arrow] (m-5-2) -- node[anchor=north,below,xshift=0.3em,yshift=-0.1em]
    {Thm. \ref{ThmLSRC}} (m-5-3);
    \draw [dashow] (m-6-1.west) to[bend left=6] node[anchor=west,right]{Thm. \ref{thm:Lupu}} (m-2-1.west);
    \draw [dashow] (m-6-2.west) to[bend left=6]  node[anchor=east,left]{Thm. \ref{PropIsoLupuLoops}}(m-2-2.west);
    \draw [dashow] (m-6-2.east) to[bend right=5.5] node[anchor=west,right]{Thm. \ref{ThmIsoLeJan}} (m-1-2.east);
\end{tikzpicture}
\label{diagram}
\end{figure}

\newpage

\section{Inversion of the signed isomorphism}
\label{sec:inversion}
In \cite{SabotTarres2015RK}, Sabot and Tarrès give a new proof of the generalized second Ray-Knight theorem together with a construction that inverts the coupling between the square of a GFF conditional its value at a vertex $x_{0}$ and the excursions of the jump process $X$ from and to $x_{0}$. 
In this paper we are interested in inverting the coupling of Theorem \ref{Lupu} with the signed GFF : more precisely, we want to describe the law of
$(X_t)_{0\le t\le \tau_u^{x_0}}$ conditional on $\varphi^{(u)}$. 

We present in Section \ref{sec_Poisson}
an inversion involving an extra Poisson process. We provide in Section \ref{sec_dicr_time} a discrete-time description of the process and in Section \ref{sec_jump} an alternative description via jump rates.  Sections \ref{sec:lejaninv} 
%and \ref{sec:coupinv} are respectively 
is dedicated to a signed inversion of Le Jan's isomorphism for loop-soups.
%, and to an inversion of the coupling of random current with FK-Ising.

\subsection{A description via an extra Poisson point process}\label{sec_Poisson}
%In this section, we give an explicit description of the law of $((X_{s}))_{s\le \tau_{u}^{x_{0}}}$ conditional on the signed GFF $(\varphi^{(u)})$.

%More start by a description of the process. 
Let $(\check \varphi_x)_{x\in V}$ be a real function on $V$ such that
$\check\varphi_{x_0}=+\sqrt{2u}$ for some $u>0$. Set 
$$
\check \Phi_x=\vert\check\varphi_x\vert,~~
\sigma_x=\operatorname{sign}(\check\varphi_x).
$$
We define a self-interacting process $(\check X_t, (\check n_e(t))_{e\in E})$ living on $V\times {\mathbb{N}}^E$ as follows.
The process $\check X$ starts at $\check X(0)=x_0$.
For $t\ge 0$, we set
$$
\check\Phi_x(t)=\sqrt{(\check\Phi_x)^2-2\check\ell_x(t)},~\forall x\in V,\qquad
J_e(\check\Phi(t))=W_e \check\Phi_{e_-}(t)\check\Phi_{e_+}(t), ~ \forall e\in E.
$$
where $\check\ell_x(t)=\int_0^t\indic_{\{\check X_s=x\}}ds$ is the local time of the process $\check X$ up to time $t$.
%$$
%J_e(\Phi(t))=W_e \Phi_{e_-}(t)\Phi_{e_+}(t).
%$$
Let $(N_e(v))_{v\ge 0}$ be an independent Poisson Point Processes on $\R_+$ with intensity 1, for each edge $e\in E$.
We set
$$
\check n_e(t)=
\begin{cases} 
N_e(2 J_e(\check\Phi(t))), &\hbox{ if } \sigma_{e_-}\sigma_{e_+}=+1,
\\
0, &\hbox{ if } \sigma_{e_-}\sigma_{e_+}=-1.
\end{cases}
$$
Given $(n_{e})_{e\in E}\in \N^{E}$ non-negative integer weights on edges, we will denote
\begin{displaymath}
\ccc(n)=\lbrace e\in E\vert n_{e}>0\rbrace.
\end{displaymath}
We consider the edges in $\ccc(n)$ as ``open'', and they naturally induce clusters. So $\ccc(\check n(t))\subset E$ denotes the configuration of edges such that $\check n_e(t)>0$.
As time increases, the interaction parameters $J_{e}(\check\Phi(t))$ decreases for the edges neighboring $\check X_t$, and at some random times $\check n_e(t)$ may drop
by 1.
The process $(\check X_t)_{t\ge 0}$ is defined as the process that jumps only at the times when one of the $\check n_e(t)$ drops by 1, as follows:
\begin{itemize}
\item
if $\check n_e(t)$ decreases by 1 at time $t$, but does not create a new cluster in $\ccc(\check n(t))$, then $\check X_t$ crosses the edge
$e$ with probability ${1/2}$ or does not move with probability ${1/2}$;
\item
if $\check n_e(t)$ decreases by 1 at time $t$, and does create a new cluster in $\ccc(\check n(t))$,
 then $\check X_t$ moves/or stays with probability 1 on the unique extremity
of $e$ which is in the cluster of the origin $x_0$ in the new configuration.
\end{itemize}
We set
$$
\check T:=\inf\{t\ge 0\vert\exists x\in V,~\text{s. t.}~ \check\Phi_x(t)=0\},
$$
clearly, the process is well-defined up to time $\check T$.
\begin{proposition}
For all $0\le t\le \check T$, $\check X_t$ is in the connected component of $x_0$ of the configuration $\ccc(\check n(t))$. If $V$ is finite,
the process ends at $x_0$, i.e. $\check X_{\check T}=x_0$.
\end{proposition}

\begin{theorem}
\label{thm-Poisson}
Assume that $V$ is finite.
With the notation of Theorem \ref{Lupu}, conditional on 
$\varphi^{(u)}=\check\varphi$, $(X_{t})_{t\le \tau_{u}^{x_{0}}}$ has the law
of $(\check X_{\check T-t})_{0\le t\le \check T}$.

Moreover, conditional on  $\varphi^{(u)}=\check\varphi$, $(\varphi^{(0)},\mathcal{O}(\varphi^{(0)}))$ has the law of
$(\sigma'\check\Phi(\check T), \ccc(\check n(\check T)))$ where 
$(\sigma'_x)_{x\in V}\in \lbrace -1,+1\rbrace^{V}$ are random spins sampled uniformly independently on 
each cluster induced by $\ccc(\check n(\check T))$,
with the condition that $\sigma'_{x_0}=+1$.
% the component of $x_0$ has a + sign.

If $V$ is infinite, then $P_{\varphi}^{\lbrace x_{0}\rbrace, \sqrt{2u}}$-a.s.,
$\check X_t$ (with the initial condition $\check\varphi=\varphi^{(u)}$)
ends at $x_0$, i.e. $\check T<+\infty$ and $\check X_{\check T}=x_0$. 
All previous conclusions for the finite case still hold.
\end{theorem}

\begin{remark}
For the process above, one could also take 
$(\check{\Phi}_{x})_{x\in V}$ deterministic and 
$(\sigma_{x})_{x\in V}$ random, distributed as an Ising model with interaction constants
$J_{e}(\check{\Phi}(0))$. Then the process
$(\check X_{t},\check{\Phi}(t))_{0\le t\le \check T}$, averaged out by the law of $(\sigma_{x})_{x\in V}$ and the evolution of $(\check n_e(t))_{e\in E}$, is exactly the same as the process introduced in
\cite{SabotTarres2015RK}, inverting the Ray-Knight identity for the square of the GFF (without the sign). Indeed, both processes (here and in \cite{SabotTarres2015RK}), when we average out by 
$\check{\Phi}(0)=\check{\Phi}$ random, distributed as 
$\vert\varphi^{(u)}\vert$ under 
$P_{\varphi}^{\lbrace x_{0}\rbrace,\sqrt{2u}}$, 
give us in law
$(X_{\tau^{x_{0}}_{u}-t}, 
\sqrt{\varphi^{(0) 2}_{x}+2\ell_{x}(\tau^{x_{0}}_{u})-
2\ell_{x}(t)})_{x\in V, 0\le t\le \tau^{x_{0}}_{u}}$
under
$\mathbb{P}_{x_{0}}(\cdot \vert \tau_{u}^{x_{0}}<\zeta)\otimes P^{\lbrace x_{0}\rbrace,0}_{\varphi}$. To conclude we use the fact that the law of
$(\check X_{t},\check{\Phi}(t))_{0\le t\le \check T}$ is continuous with respect to
$\check{\Phi}(0)$.
\end{remark}

\subsection{Discrete time description of the process}
\label{sec_dicr_time}

We give a discrete time description of the process
$(\check X_t, (\check n_e(t))_{e\in E})$
that appears in the previous section.
Let $t_{0}=0$ and $0<t_{1}<\dots<t_{j}$ be the stopping times when one of the
stacks $\check n_e(t)$ decreases by $1$, where $t_{j}$ is the time when one of the stacks is completely depleted. It is elementary to check the following:

\begin{proposition}
\label{PropDiscrTime}
The discrete time process
$(\check X_{t_{i}}, (\check n_e(t_{i}))_{e\in E})_{0\leq i\leq j}$ is a stopped Markov process. The transition from time $i-1$ to $i$ is the following:
\begin{itemize}
\item first chose $e$ an edge adjacent to the vertex $\check X_{t_{i-1}}$
according to a probability proportional to $\check n_e(t_{i-1})$;
\item decrease the stack $\check n_e(t_{i-1})$ by 1;
\item
if decreasing $\check n_e(t_{i-1})$ by 1 does not create a new cluster in 
$ \ccc(\check n(t_{i-1}))$, then $\check X_{t_{i-1}}$ crosses the edge
$e$ with probability ${1/2}$ or does not move with probability ${1/2}$;
\item
if decreasing $\check n_e(t_{i-1})$ by 1 does create a new cluster in 
$\ccc(\check n(t_{i-1}))$,
 then $\check X_{t_{i-1}}$ moves/or stays with probability 1 on the unique extremity of $e$ which is in the cluster of the origin $x_0$ in the new configuration.
\end{itemize}
\end{proposition}

\subsection{An alternative description via jump rates}\label{sec_jump}
We provide an alternative description of the process 
$(\check X_t,\ccc(\check n(t)))$ 
that appears in Section \ref{sec_Poisson}.
We will use this description in
\cite{LST2018InversionLine} by passing it to a fine mesh limit to obtain a process inverting the Ray-Knight identity for a Brownian motion on $\mathbb{R}$.

We will denote
$\check \ccc(t)=\ccc(\check n(t))$, since below we will not have access to the knowledge of $\check n(t)$, only to that of
$\ccc(\check n(t))$.

\begin{proposition}\label{prop-jump}
The process $(\check X_t,\check \ccc(t))$ defined in Section \ref{sec_Poisson} can be alternatively described by its jump rates :
conditional on its past at time $t$, if $\check X_t=x$, $y\sim x$ and $\lbrace x,y\rbrace\in \check \ccc(t)$, then
\begin{itemize}
\item[(1)]  $\check X$ jumps to $y$ without modification of 
$\check \ccc(t)$ at rate
\begin{displaymath}
W_{x,y}\dfrac{\check\Phi_{y}(t)}{\check\Phi_{x}(t)};
\end{displaymath}
\item[(2)]  the edge $\lbrace x,y\rbrace$ is closed in 
$\check \ccc(t)$ at rate
\begin{displaymath}
2W_{x,y}\dfrac{\check\Phi_{y}(t)}{\check\Phi_{x}(t)}
\left(e^{2W_{x,y}\check\Phi_{x}(t)\check\Phi_{y}(t)}-1\right)^{-1}
\end{displaymath}
and, conditional on that last event:

- if $y$ is connected to
$x$ in the configuration $\check \ccc(t)\setminus\{x,y\}$, 
then $\check X$ instantaneously jumps to $y$ with probability $1/2$ and stays at $x$ with probability $1/2$;

- otherwise $\check X_t$ moves/or stays with probability 1 on the unique extremity
of $\{x,y\}$ which is in the cluster of the origin $x_0$ in the new configuration.
\end{itemize}
\end{proposition}
\begin{remark}
It is clear from this description that the joint process 
$(\check X_t,\check \ccc(t), \check \Phi(t))$ is Markov process, and well defined up to the time
$$
\check T:=\inf\{t\ge 0\vert \exists x\in V, ~\text{s.t.}~\check\Phi_x(t)=0\}.
$$
\end{remark}

\begin{remark}
One can also retrieve the process in Section \ref{sec_Poisson} from the representation in Proposition \ref{prop-jump} as follows.
Consider the representation of Proposition \ref{prop-jump} on the graph where each edge $e$ is replaced by a large number $N$ of 
parallel edges with conductance $W_e/N$. Consider now $\check n^{(N)}_{x,y}(t)$ the number of parallel edges that are open in the configuration
$\check \ccc^{(N)}(t)$ between $x$ and $y$. Then, when $N\to\infty$, $(\check n^{(N)}(t))_{t\ge0}$,  converges in law to 
$(\check n(t))_{t\ge0}$,  defined in Section \ref{sec_Poisson}.
We will not detail this, but roughly, this corresponds to approximation of Poisson r.v. by binomial r.v. 
\end{remark}

\begin{proof}[Proof of Proposition \ref{prop-jump}]
Here $(\check\fff_{t})_{t\ge 0}$ will denote the natural filtration of
$(\check X_{t\wedge \check T},
\check \ccc(t\wedge \check T))_{t\geq 0}$.

Assume $\check X_t=x$, fix $y\sim x$ and let $e=\{x,y\}$. Recall that 
$\{x,y\}\in\check \ccc(t)$ iff $\check n_e(t)\ge 1$.
Let be
\begin{displaymath}
J_{e}^{x}(t,\Delta t)=
W_{x,y} \sqrt{\check\Phi_{x}(t)^{2}-2\Delta t}
\,\check\Phi_{y}(t),
\end{displaymath}
which is the value of 
$J_{e}(\Phi(t+\Delta t))$ on the event
\begin{displaymath}
\{X_{t}=x,\, n(t+\Delta t)=n(t),\,
W_{x,y} \sqrt{\check\Phi_{x}(t)^{2}-2\Delta t}
\,\check\Phi_{y}(t)\geq 0\}.
\end{displaymath}

Let us first prove (1):
\begin{align*}
&\Pb\left(\text{On $[t,t+\Delta t]$ $\check X$ first jumps from $x$ to $y$ without modifying }
\check \ccc|
\check X_{t}=x,\,e\in
\check \ccc(t),\,
\check\fff_{t}\right)\\
&=\frac{1}{2}\Pb\left(\exists s\in [t,t+\Delta t],
\check n(t)-\check n(s)=\delta_{e}
,\,\check n_e(s)\ge1\,|\,\check X_{t}=x,\,\check n_e(t)\ge 1
,\,\check\fff_{t}
\right)+o(\Delta t)\\
&=
\frac{1}{2}\Pb\left(N_{e}(2 J_e(\check\Phi(t))\ge 2,\,
N_{e}(2 J_e(\check\Phi(t))-
N_{e}(2 J_{e}^{x}(t,\Delta t))
\ge 1\,|\,\check X_{t}=x,\,
\check n_{e}(t)\ge 1
,\,\check\fff_{t}
\right)+o(\Delta t)
\\
&=
\frac{1}{2}\Pb\left(N_{e}(2 J_{e}^{x}(t,\Delta t))\ge 1,\,
N_{e}(2 J_e(\check\Phi(t))-
N_{e}(2 J_{e}^{x}(t,\Delta t))
\ge 1\,|\,\check X_{t}=x,\,
\check n_{e}(t)\ge 1
,\,\check\fff_{t}
\right)+o(\Delta t)
\\
&=
\frac{1}{2}
\dfrac{1-e^{-2 J_{e}^{x}(t,\Delta t)}}
{1-e^{-2 J_e(\check\Phi(t))}}
(1-e^{-2(J_e(\check\Phi(t))-J_{e}^{x}(t,\Delta t))})
+o(\Delta t)
\\
&=J_e(\check\Phi(t))-J_{e}^{x}(t,\Delta t) + o(\Delta t)
=W_{xy}\dfrac{\check\Phi_{y}(t)}{\check\Phi_{x}(t)}\Delta t+o(\Delta t).
\end{align*}

Similarly, (2) follows from the following computation:
\begin{align*}
&\Pb\left(\text{On $[t,t+\Delta t]$ first the edge $e$ is closed in $\check \ccc$}\,|
\,\check X_{t}=x,\,e\in\check \ccc(t),\,
\check \fff_{t}\right)\\
&=\Pb\left(N_{e}(2 J_{e}^{x}(t,\Delta t))=0,\,
N_{e}(2 J_e(\check\Phi(t))=1
\,|\,\check X_{t}=x,\,\check n_e(t)\ge1,\,
\check \fff_{t}\right)
+o(\Delta t)
\\&=
\dfrac{e^{-2 J_{e}^{x}(t,\Delta t)}
2(J_e(\check\Phi(t))-J_{e}^{x}(t,\Delta t))
e^{-2(J_e(\check\Phi(t))-J_{e}^{x}(t,\Delta t))}}
{1-e^{-2 J_e(\check\Phi(t))}}
+o(\Delta t)
\\&=
\dfrac{2(J_e(\check\Phi(t))-J_{e}^{x}(t,\Delta t))}
{e^{2 J_e(\check\Phi(t))}-1}
+o(\Delta t)
\\&=
2W_{x,y}\dfrac{\check\Phi_{y}(t)}{\check\Phi_{x}(t)}
\left(e^{2W_{x,y}\check\Phi_{x}(t)\check\Phi_{y}(t)}-1\right)^{-1}
+o(\Delta t).
\qedhere
\end{align*}
\end{proof}

We easily deduce from the Proposition \ref{prop-jump} and Theorem \ref{thm-Poisson2} the following alternative inversion of the coupling in Theorem \ref{Lupu}.
\begin{theorem}\label{thm-jump-rates}
With the notation of Theorem \ref{Lupu}, conditional on 
$(\varphi^{(u)},\mathcal{O}_{u})$, $(X_{t})_{t\le \tau_{u}^{x_{0}}}$ has the law
of self-interacting process $(\check X_{\check T-t})_{0\le t\le \check T}$ defined by jump rates of Proposition \ref{prop-jump}
starting with 
$$
\check \Phi_x=\sqrt{(\varphi_{x}^{(0)})^2+2\ell_{x}(\tau_{u}^{x_{0}})} \hbox{ and } 
\check\ccc(0)=\mathcal{O}_{u}.
$$
Moreover $(\varphi^{(0)},\mathcal{O}(\varphi^{(0)}))$ has the same law as
$(\sigma'\check\Phi(\check T), \check\ccc(\check T))$ where 
$(\sigma'_x)_{x\in V}$ is a configuration of signs obtained by picking a sign at random independently on
each connected component of $\check\ccc(\check T)$, with the condition that the component of $x_0$ has a + sign.

\end{theorem}

%\subsection{Back to the ``magnetized'' reverse Vertex-Reinforced Jump Process in \cite{SabotTarres2015RK}}
%\label{sec:back}
\subsection{Inversion of Lupu's isomorphism for loop-soup}
\label{sec:lejaninv}
Let us first recall how the loops in $\mathcal{L}_{\alpha}$ are connected to the excursions of the jump process $X$. We refer to \cite{LeJan2011Loops} for details. $\textbf{G}$ is the Green's function. 
$L_{x_{0}}(\mathcal{L}_{\alpha})$ follows a $\Gamma(\alpha, \textbf{G}(x_{0},x_{0}))$ distribution, that is to say
$L_{x_{0}}(\mathcal{L}_{\alpha})/\textbf{G}(x_{0},x_{0})$
follows a Gamma distribution 
$\Gamma(\alpha, 1)$ with density
\begin{displaymath}
\indic_{\{r>0\}}\dfrac{1}{\Gamma(\alpha)}
r^{\alpha -1} e^{-r} dr.
\end{displaymath}
As a process in $\alpha$, where one drops independent loops as the intensity parameter $\alpha$ increases,
$(L_{x_{0}}(\mathcal{L}_{\alpha})/\textbf{G}(x_{0},x_{0}))_{\alpha\geq 0}$ is a pure jump 
Gamma subordinator with Lévy measure
\begin{displaymath}
d\Lambda (r) = 
\indic_{\{r>0\}}\dfrac{1}{\Gamma(\alpha)}\dfrac{ e^{-r}}{r}
 dr.
\end{displaymath}
Given such a Gamma subordinator
$(R(\alpha))_{\alpha\geq 0}$ with
$R(0)=0$, then for all $\alpha>0$, the marginal 
$R(\alpha)$ follows a $\Gamma(\alpha, 1)$ distribution. Moreover, the normalized family of jump sizes
\begin{displaymath}
\left(\dfrac{R(a)-R(a^{-})}{R(\alpha)}\right)
_{0\leq a\leq \alpha, R(a)\neq R(a^{-})}
\end{displaymath}
is independent of $R(\alpha)$ and has the law of a 
Poisson-Dirichlet partition $PD(0,\alpha)$ of 
$[0,1]$. The above may be taken as a definition of
$PD(0,\alpha)$. It is a random infinite countable family of positive reals summing to $1$. For more on Poisson-Dirichlet partitions, we refer to \cite{Pitman2006StFlour}.

\begin{proposition}[From excursions to loops]
\label{PropPD}
Let $\alpha>0$ and $x_{0}\in V$. 
$L_{x_{0}}(\mathcal{L}_{\alpha})$ is distributed according to a Gamma
$\Gamma(\alpha, \textbf{G}(x_{0},x_{0}))$ law, where 
$\textbf{G}$ is the Green's function. Let $u>0$, and consider the path $(X_{t})_{0\leq t\leq \tau_{u}^{x_{0}}}$ conditional on $\tau_{u}^{x_{0}}<\zeta$. Let $(Y_{j})_{j\geq 1}$ be an independent Poisson-Dirichlet partition $PD(0,\alpha)$ of $[0,1]$ (so that $\sum_{j\geq 1}Y_{j}=1$). Let $S_{0}=0$ and
\begin{displaymath}
S_{j}=\sum_{i=1}^{j}Y_{i}.
\end{displaymath}
Let
\begin{displaymath}
\tau_{j}= \tau_{u S_{j}}^{x_{0}}.
\end{displaymath}
Consider the family of paths
\begin{displaymath}
\left((X_{\tau_{j-1}+t})_{0\leq t\leq \tau_{j}-\tau_{j-1}}\right)_{j\geq 1}.
\end{displaymath}
It is a countable family of loops rooted in $x_{0}$. It has the same law as the family of all the loops in $\mathcal{L}_{\alpha}$ that visit $x_{0}$, conditional on $L_{x_0}(\mathcal{L}_{\alpha})=u$.
\end{proposition}

Next we describe how to invert the discrete version for Lupu's isomorphism Theorem \ref{PropIsoLupuLoops} for the loop-soup in the same way as in Theorem \ref{thm-Poisson}. The idea is to define an arbitrary order on the vertices, $(x_{i})_{1\leq i\leq\vert V\vert}$. Then, starting from a signed
GFF, run the inverting process introduced previously starting from $x_{1}$, up to exhausting the field in $x_{1}$. This would produce a path from $x_{1}$ to $x_{1}$, which is conditionally distributed like all the loops in
$\mathcal{L}_{1/2}$ that visit $x_{1}$, glued together. By partitioning this path according to the procedure described in
Proposition \ref{PropPD}, one recovers all the loops visiting
$x_{1}$. Then one continues with the remaining field, which is $0$ in
$x_{1}$ and has smaller FK-Ising clusters than the initial one, and runs the inverting process starting from $x_{2}$, in order to get all the loops that visit $x_{2}$ but not $x_{1}$. Then one iterates. At each step, one gets all the loops that visit $x_{i}$, but none of
$x_{1},\dots, x_{i-1}$. In what follows we describe this more formally.

Let $(\check \varphi_x)_{x\in V}$ be a real function on $V$. Set 
$$
\check \Phi_x=\vert\check\varphi_x\vert, ~~\sigma_x=\operatorname{sign}(\check\varphi_x).
$$
Let $(x_{i})_{1\leq i\leq\vert V\vert}$ be an enumeration of $V$ (which may be infinite).
We define by induction on $i$ the self interacting processes
$((\check X_{i,t})_{1\leq i\leq\vert V\vert}, 
(\check n_e(t))_{e\in E})$. 
$\check{T}_{i}$ will denote the end-time for $\check X_{i,t}$, and
$\check{T}^{+}_{i}=\sum_{1\leq j\leq i}\check{T}_{j}$.
By definition, $\check{T}^{+}_{0}=0$.
$L_{x}(t)$ will denote
\begin{displaymath}
L_{x}(t):=\sum_{1\leq i\leq\vert V\vert}
\check{\ell}_{x}(i,0\vee(t-\check{T}^{+}_{i})),
\end{displaymath}
where $\check{\ell}_{x}(i,t)$ are the occupation times for
$\check X_{i,t}$.
For $t\ge 0$, we set
$$
\check\Phi_x(t)=\sqrt{(\check\Phi_x)^2-2L_x(t)},~\forall x\in V,
\qquad
J_e(\check\Phi(t))=W_e \check\Phi_{e_-}(t)\check\Phi_{e_+}(t), \forall e\in E.
$$
The end-times $\check{T}_{i}$ are defined by induction as
\begin{displaymath}
\check{T}_{i}=\inf\lbrace t\geq 0\vert 
\check{\Phi}_{\check{X}_{i,t}}(t+\check{T}^{+}_{i-1})=0\rbrace.
\end{displaymath}
Let $(N_e(v))_{v\ge 0}$ be independent Poisson Point Processes on $\R_+$ with intensity 1, for each edge $e\in E$.
We set
$$
\check n_e(t)=
\begin{cases} 
N_e(2 J_e(\check\Phi(t))), &\hbox{ if } \sigma_{e_-}\sigma_{e_+}=+1,
\\
0, &\hbox{ if } \sigma_{e_-}\sigma_{e_+}=-1.
\end{cases}
$$
We also denote by $\ccc(\check n(t))\subset E$ the configuration of edges such that $\check n_e(t)>0$.
$\check X_{i,t}$ starts at $x_{i}$.
For $t\in[\check{T}^{+}_{i-1},\check{T}^{+}_{i}]$,
\begin{itemize}
\item
if $\check n_e(t)$ decreases by 1 at time $t$, but does not create a new cluster in $\ccc(\check n(t))$, then $\check X_{i,t-\check{T}^{+}_{i-1}}$ crosses the edge
$e$ with probability ${1/2}$ or does not move with probability ${1/2}$;
\item
if $\check n_e(t)$ decreases by 1 at time $t$, and does create a new cluster in $\ccc(\check n(t))$,
 then $\check X_{i,t-\check{T}^{+}_{i-1}}$ moves/or stays with probability 1 on the unique extremity
of $e$ which is in the cluster of the origin $x_i$ in the new configuration.
\end{itemize}
%In this evolution, $\check{T}_{i}^{+}$ denotes the first time
%$\check{\Phi}_{x_{i}}(t)$ attains $0$.

By induction, using Theorem \ref{thm-Poisson}, we deduce the following:

\begin{theorem}
\label{ThmPoissonLoopSoup}
Let $\varphi$ be a GFF on $\mathcal{G}$ with the law $P_{\varphi}$.
If one sets $\check{\varphi}=\varphi$ in the preceding construction, then
for all $i\in \lbrace 1,\dots,\vert V\vert\rbrace$,
$\check{T}_{i}<+\infty$, 
$\check{X}_{i,\check{T}_{i}} = x_{i}$ and the
path $(\check{X}_{i,t})_{t\leq\check{T}_{i}}$ has the same law as a concatenation in $x_{i}$ of all the loops in a loop-soup
$\mathcal{L}_{1/2}$ that visit $x_{i}$, but none of the
$x_{1},\dots,x_{i-1}$. To retrieve the loops out of each path
$(\check{X}_{i,t})_{t\leq\check{T}_{i}}$, one has to partition it according to
a Poisson-Dirichlet partition as in Proposition \ref{PropPD}.
The coupling between the GFF $\varphi$ and the loop-soup obtained from
$((\check X_{i,t})_{1\leq i\leq\vert V\vert}, 
(\check n_e(t))_{e\in E})$ is the same as in Theorem
\ref{PropIsoLupuLoops}.
\end{theorem}

\begin{remark}
\label{RemInvCurrentFKIsing}
One could consider the discrete time version of the procedure described in
Theorem \ref{ThmPoissonLoopSoup}, and look only at the total number of times $\hat{n}_{e}$ an edge $e$ is visited by the trajectories constructed, without distinguishing the directions. Than $(\hat{n}_{e})_{e\in E}$ is a current, and its conditional distribution given $\check{\mathcal{C}}(0)$ is the same as the conditional distribution of a random current given an FK-Ising cluster when both are coupled as in Theorem \ref{RCFKIsing}.
\end{remark}

\begin{corollary}
\label{CorPoissonLoopSoup2}
Let $\varphi$ be a GFF on $\mathcal{G}$ with the law $P_{\varphi}$. Let $(\mathfrak{n}_{e}(\varphi))_{e\in E}$
be a family of r.v., distributed conditional on $\phi$ as independent Poisson r.v., each one with mean
$W_{e}\varphi_{e_{-}}\varphi_{e_{+}}$ if
$\varphi_{e_{-}}\varphi_{e_{+}}>0$, $0$ otherwise. One can couple $\varphi$, $(\mathfrak{n}_{e}(\varphi))_{e\in E}$ and a loop-soup $\mathcal{L}_{1/2}$ such that the coupling between $\mathcal{L}_{1/2}$ and $\varphi$ is that of
Theorem \ref{PropIsoLupuLoops}, and moreover, a.s. for every $e\in E$,
$\mathsf{N}_{e}(\mathcal{L}_{1/2})
\in\lbrace\mathfrak{n}_{e}(\varphi)-1,
\mathfrak{n}_{e}(\varphi),\mathfrak{n}_{e}(\varphi)+1\rbrace$.
\end{corollary}

\begin{proof}
We use the construction of Theorem \ref{ThmPoissonLoopSoup}. Conditional on $\varphi$, we have independent Poisson stacks with mean
$\indic_{\{\varphi_{e_{-}}\varphi_{e_{+}}>0\}}2 
W_{e}\varphi_{e_{-}}\varphi_{e_{+}}$
and each time we unpile a stack, we chose with probability $1/2$ to jump (which gives a Poisson r.v. with mean
$\indic_{\{\varphi_{e_{-}}\varphi_{e_{+}}>0\}} 
W_{e}\varphi_{e_{-}}\varphi_{e_{+}}$), except possibly when the stack is reduced to $1$, when our choice might be constrained
(which gives $\pm 1$).
\end{proof}

\section{Proof of theorem \ref{thm-Poisson} }
\label{sec:proof}
\subsection{Case of finite graph without killing measure}
\label{sec:pfinite}
Here we will assume that $V$ is finite and that the killing measure
$\kappa\equiv 0$.

In order to prove Theorem \ref{thm-Poisson}, we first enlarge the state space of the process $(X_t)_{t\ge 0}$. We define a process
$(X_t,(n_e(t)))_{t\ge 0}$ living on the space $V\times {\mathbb N}^E$ as follows. Let 
$\varphi^{(0)}\sim P_{\varphi}^{\{x_0\},0}$ be a GFF pinned at $x_0$.
%$\varphi\in \R^V$ be a function on $V$ such that $\varphi_{x_0}=0$.
Let $\sigma_x=\hbox{sign}(\varphi^{(0)}_x)$ be the signs of the GFF with the convention that $\sigma_{x_0}=+1$.
The process $(X_t)_{t\ge 0}$ is as usual the Markov Jump process starting at $x_0$ with jump rates $(W_e)_{e\in E}$. We set
\begin{equation}
\label{Phi-J}
\Phi_x=\vert\varphi^{(0)}_x\vert, ~~\Phi_{x}(t)=\sqrt{\Phi_x^2+2\ell_x(t)}, ~\forall x\in V, \qquad J_e(\Phi(t))=W_e \Phi_{e_{-}}(t)\Phi_{e_{+}}(t), ~\forall e\in E.
\end{equation}
The initial values $(n_e(0))$ are choosen independently on each edge with distribution
\begin{equation}
\label{Eq_n_0}
n_e(0)\sim
\begin{cases}
0& \hbox{ if $\sigma_{e_-}\sigma_{e_+}=-1,$}
\\
\mathcal{P}(2J_e(\Phi))& \hbox{ if $\sigma_{e_-}\sigma_{e_+}=+1,$}
\end{cases}
\end{equation}
where ${\mathcal{P}}(2J_e(\Phi))$ is a Poisson random variable with parameter $2J_e(\Phi)$. Let $((N_e(v))_{v\ge 0})_{e\in E}$ be independent Poisson point processes 
on $\R_+$ with intensity 1. We define the process $(n_e(t))$ by
$$
n_e(t)=n_e(0)+N_e(J_e(\Phi(t)))-N_e(J_e(\Phi))+K_e(t),
$$
where $K_e(t)$ is the number of crossings of the edge $e$ by the Markov jump process $X$ before time $t$.
\begin{remark}
Note that compared to the process defined in Section \ref{sec_Poisson}, the speed of the Poisson process is related to $J_e(\Phi(t))$ and not $2J_e(\Phi(t))$.
\end{remark}
We recall that with our notations,
$$
\ccc(n(t))=\{e\in E\vert n_e(t)>0\}.
$$
Recall also that $\tau_u^{x_0}=\inf\{t\ge 0\vert\ell_{x_0}(t)=u\}$ for $u>0$. 
%To simplify notation, we will write $\tau_u$ for $\tau_u^{x_0}$ in the sequel.
We define $\varphi^{(u)}$ by
$$
\varphi^{(u)}_x=\sigma_x\Phi(\tau_{u}^{x_{0}}), ~\forall x\in V,
$$
where  $(\sigma_x)_{x\in V}\in \lbrace -1,+1\rbrace^{V}$ are random spins sampled uniformly independently on 
each cluster induced by $\ccc(n(\tau_{u}^{x_{0}}))$ with the condition that $\sigma_{x_0}=+1$.

\begin{lemma}
\label{end-distrib}
The random family $(\varphi^{(0)}, \ccc(n(0)), \varphi^{(u)}, \ccc(n(\tau_u^{x_0})))$ thus defined has the same distribution
 as $(\varphi^{(0)}, \ooo(\varphi^{(0)}), \varphi^{(u)}, \ooo_u)$ defined in Theorem \ref{Lupu}.
\end{lemma}
\begin{proof}
It is clear from construction, that $\ccc(n(0))$ has the same law as 
$\ooo(\varphi^{(0)})$ (cf Definition \ref{def_FK-Ising}), the FK-Ising configuration coupled with the signs of
$\varphi^{(0)}$ as in Theorem \ref{FK-Ising}. Indeed, for each edge $e\in E$ such that $\varphi^{(0)}_{e_-}\varphi^{(0)}_{e_+}>0$, the probability that
$n_e(0)>0$ is $1-e^{-2J_e(\Phi)}$.
Moreover, conditional on $\ccc(n(0))=\ooo(\varphi^{(0)})$, 
$\ccc(n(\tau_u^{x_0}))$ has the same law as $\ooo_u$ defined in Theorem \ref{Lupu}. Indeed, $\ccc(n(\tau_u^{x_0}))$
is the union of the set $\ccc(n(0))$, the set of edges crossed by the process $(X_u)_{u\le \tau_u^{x_0}}$, and the additional edges such that $N_e(J_e(\tau_u^{x_0}))-N_e(J_e(\Phi))>0$.
Clearly $N_e(J_e(\tau_u^{x_0}))-N(J_e(\Phi))>0$ independently with probability $1-e^{-(J_e(\Phi(\tau_u^{x_0}))-J_e(\Phi))}$ which coincides with the probability given in
Theorem \ref{Lupu}. 
\end{proof}

We will prove the following theorem that, together with Lemma \ref{end-distrib}, contains the statements of both Theorem \ref{Lupu} and \ref{thm-Poisson}.
\begin{theorem}\label{thm-Poisson2}
%Let $(\sigma_x)\in \{\pm 1\}^V$ be be obtained by piking random signs on 
%connected components of ${\mathcal{C}}(\tau_u^{x_0})$ with the convention that the component of $x_0$ has a positive sign.
%Let
%$$
%\varphi^{(u)}_x=\sigma_x \Phi_x(\tau_u^{x_0}).
%$$
The random field $\varphi^{(u)}$ is a GFF distributed according to $P_{\varphi}^{\{x_0\},\sqrt{2u}}$.
% and ${\mathcal{C}}(\tau_u^{x_0})$ is coupled with $(\sigma_x)$ as 
%FK-Ising configuration with interaction $(J_e(\vert \varphi^{(u)}))$. 
Moreover, conditional on $\varphi^{(u)}=\check \varphi$, the process 
$$(X_{t},(n_{e}(t))_{e\in E})_{t\le \tau_u^{x_0}}$$ 
has the law of the process $(\check X_{\check T-t },(\check n_e(\check T -t))_{e\in E})_{t\le \check T}$
described in Section \ref{sec_Poisson}.
\end{theorem} 

\begin{proof}
Before proceeding to the proof of the theorem, we will briefly outline
our method. We have a Markov process
$(X_{t},\Phi(t),n(t))$, and would like to find another Markov process
$(\check X_{t},\check\Phi(t),\check n(t))$ such that the latter has the law of
$(X_{\tau_u^{x_0}-t},\Phi(\tau_u^{x_0}-t),n(\tau_u^{x_0}-t))$ in the particular case when the entrance (initial) distribution of
$(X_{0},\Phi(0),n(0))$ is given by 
$(x_{0},\vert\varphi^{(0)}\vert,n(0))$, $n(0)$ given by
\eqref{Eq_n_0}. To this end, we will first introduce an intermediate Markov process
$(\bar X_{t},\bar\Phi(t),\bar n(t))$ which will correspond to
$(X_{\tau_u^{x_0}-t},\Phi(\tau_u^{x_0}-t),n(\tau_u^{x_0}-t))$ in case when the entrance distribution of $(X_{0},\Phi(0),n(0))$ is not the one we are interested in, but given by $X_{0}=x_{0}$ and $(\Phi(0),n(0))$
following the product measure (with infinite total mass)
\begin{displaymath}
\sum_{n\in \N^{E}}\int d\Phi F(\Phi,n).
\end{displaymath}
We do that because $(\bar X_{t},\bar\Phi(t),\bar n(t))$ is simpler, and in particular $\bar X_{t}$ is a Markov jump process with jump rates
$(W_{e})_{e\in E}$ (just as $X_{t}$) which does not interact with
$(\bar\Phi(t),\bar n(t))$. In this way 
$(\bar X_{t},\bar\Phi(t),\bar n(t))$ and 
$(\check X_{t},\check\Phi(t),\check n(t))$ correspond to
$(X_{\tau_u^{x_0}-t},\Phi(\tau_u^{x_0}-t),n(\tau_u^{x_0}-t))$ for two different entrance distributions of 
$(X_{0},\Phi(0),n(0))$, and we show that
$(\check X_{t},\check\Phi(t),\check n(t))$ is absolutely continuous with respect to
$(\bar X_{t},\bar\Phi(t),\bar n(t))$ and identify the corresponding
Radon-Nikodym derivative $\bar M_{t\wedge \bar T}/\bar M_{0}$. Out of this we further identify $(\check X_{t},\check\Phi(t),\check n(t))$ as
a Doob's h-transform (see \cite{ChungWalsh05MP}, Chapter 11) of 
$(\bar X_{t},\bar\Phi(t),\bar n(t))$. Then we obtain the infinitesimal generator of 
$(\check X_{t},\check\Phi(t),\check n(t))$
as a conjugate of the infinitesimal generator of
$(\bar X_{t},\bar\Phi(t),\bar n(t))$.

Our proof is cut in four steps. In step 1 we simply give an explicit formula for the law of $(\vert\varphi^{(0)}\vert, n(0))$, 
$n(0)$ given by \eqref{Eq_n_0}. In step 2 we introduce
$(\bar X_{t},\bar\Phi(t),\bar n(t))$ and prove the time reversal for the product entrance distribution of 
$(X_{0},\Phi(0),n(0))$. In step 3 we identify the Radon-Nikodym derivative of $(\check X_{t},\check\Phi(t),\check n(t))$ with respect to $(\bar X_{t},\bar\Phi(t),\bar n(t))$. In step 4 we describe 
$(\bar X_{t},\bar\Phi(t),\bar n(t))$ as 
a Doob's h-transform of $(\bar X_{t},\bar\Phi(t),\bar n(t))$ and give its the infinitesimal generator, which is that of the process introduced in Section \ref{sec_Poisson}.

\medskip

\noindent{\bf Step 1 :}
We start by a simple lemma.
\begin{lemma}\label{distrib-phi-n}
The distribution of $(\Phi:=\vert \varphi^{(0)}\vert, n(0))$ is given by the following formula for any bounded measurable test function $h$:
\begin{multline*}
\E\left(h(\Phi, n(0))\right)= \\\sum_{n\in \N^E} \int_{\R_+^{V\setminus\{x_0\}}} d\Phi  h(\Phi, n) 
%\exp\left(
e^{-\demi \sum_{x\in V} W_x\Phi_x^2-\sum_{e\in E} J_e(\Phi)}
%\right) 
\left(\prod_{e\in E}{\frac{(2J_e(\Phi))^{n_e}}{n_e!}}\right)
2^{\#c.\ccc(n)-1}.
\end{multline*}
where the integral is on the set 
$\{(\Phi_x)_{x\in V}\in \R_{+}^{V}\vert \forall x\neq x_0,~
\Phi_x>0,~\Phi_{x_0}=0\}$, and 
$$d\Phi={\frac{\prod_{x\in V\setminus\{x_0\}} d\Phi_x}{\sqrt{2\pi}^{\vert V\vert -1}}},$$
and $\#c.\ccc(n)$ is the number of clusters 
induced by the edges such that $n_e>0$.
\end{lemma}
\begin{proof}
Indeed, by construction, summing on possible signs of $\varphi^{(0)}$, we have
\begin{eqnarray}
\nonumber
&&\E\left(h(\Phi, n(0))\right)=
\\
\label{int-eee}&&
\sum_{\substack{\sigma\in \{\pm 1\}^V
\\\sigma_{x_0}=+1}}
 \sum_{\substack{n\in \N^{E}\\n\ll \sigma}}
 \int_{\R_+^{V\setminus\{x_0\}}} d\Phi  h(\Phi, n) e^{-\demi \eee(\sigma\Phi,\sigma\Phi)}\Bigg(
 \prod_{\substack{e\in E\\\sigma_{e_-}\sigma_{e_+}=+1}} {e^{-2J_e(\Phi)} (2J_e(\Phi))^{n_e}\over n_e!}\Bigg),
 \end{eqnarray}
where 
%the first sum is on the set 
%$\{(\sigma_x)_{x\in V}\in \{\pm 1\}^V \vert \sigma_{x_0}=+1\}$ 
%and the second sum is on the set of
%$\{(n_e)_{e\in E}\in \N^E \vert
%n_e=0 ~\text{if}~\sigma_{e_-}\sigma_{e_+}=-1\}$ (we write 
$n\ll \sigma$ means that $n_e$ vanishes on the edges
such that  $\sigma_{e_-}\sigma_{e_+}=-1$. Since we have,
similarly to \eqref{EqDirichlet},
\begin{eqnarray*}
\demi\eee(\sigma \Phi,\sigma \Phi)&=& \demi\sum_{x\in V} W_x \Phi_x^2-\sum_{e\in E} J_e(\Phi)\sigma_{e_-}\sigma_{e_+}
\\
&=&
 \demi\sum_{x\in V} W_x \Phi_x^2+\sum_{e\in E} J_e(\Phi) 
 -\sum_{\substack{e\in E\\\sigma_{e_-}\sigma_{e_+}=+1}} 2J_e(\Phi),
\end{eqnarray*}
we deduce that the integrand in \eqref{int-eee} is equal to
\begin{eqnarray*}
 && h(\Phi,n) e^{-\demi \eee(\sigma\Phi,\sigma \Phi)}
 \Bigg(\prod_{\substack{e\in E\\\sigma_{e_-}\sigma_{e_+}=+1}} {e^{-2J_e(\Phi)} (2J_e(\Phi))^{n_e}\over n_e!}\Bigg)
 \\
 &=&
h(\Phi,n)   e^{-\demi \eee(\sigma\Phi,\sigma \Phi)}
e^{-\sum_{e\in E, \; \sigma_{e_-}\sigma_{e_+}=+1} 2J_e(\Phi)}\left(\prod_{e\in E} {(2J_e(\Phi))^{n_e}\over n_e!}\right)
 \\
  &=&
 h(\Phi,n)  e^{-\demi\sum_{x\in V} W_x \Phi_x^2-\sum_{e\in E} J_e(\Phi)}\left(\prod_{e\in E} {(2J_e(\Phi))^{n_e}\over n_e!}\right),
\end{eqnarray*}
where we used in the first equality that $n_e=0$ on the edges such that  
$\sigma_{e_-}\sigma_{e_+}=-1$.
Thus,
\begin{multline*}
\E\left(h(\Phi, n(0))\right)=\\
\sum_{\substack{\sigma\in \{\pm 1\}^V
\\\sigma_{x_0}=+1}}
 \sum_{\substack{n\in \N^{E}\\n\ll \sigma}}
 \int_{\R_+^{V\setminus\{x_0\}}} d\Phi h(\Phi, n) e^{-\demi\sum_{x\in V} W_x \Phi_x^2-\sum_{e\in E} J_e(\Phi)}\left(\prod_{e\in E} {(2J_e(\Phi))^{n_e}\over n_e!}\right).
\end{multline*}
%where now the first sum is on the full set $(n_e)\in \N^E$ and the second sum is on the set of signs 
%$(\sigma_x)_{x\in V}$ such that $\sigma_{x_0}=+1$ and such that $\sigma_x$ is constant on connected components induced by the configuration of edges
%$\{e\in E, \; n_e>0\}$. 
Reversing the sum on $\sigma$ and $n$ and summing on the number of possible signs which are constant on  clusters induced by the configuration of edges
$\{e\in E\vert n_e>0\}$,
we deduce Lemma \ref{distrib-phi-n}.
\end{proof}
\noindent{\bf Step 2 :} We denote by $Z_t=(X_t, \Phi(t), n(t))$ the process defined previously and by
$E_{x_0, \Phi, n_0}$ its law with initial condition $(x_0, \Phi, n)$. 

We now introduce a process $\bar{Z}_t$, which is a 
``time reversal'' of the process $Z_t$. This process will be related to the 
process defined in Section \ref{sec_Poisson} in Step 4, Lemma \ref{RN}.

For $(\bar n_e)_{e\in E}\in \N^E$ and $(\bar \Phi_x)_{x\in V}$ such that
$$
\bar \Phi_{x_0}=u, ~\forall x\neq x_0,~\bar \Phi_x>0,
$$
we define the process $\bar Z_t=(\bar X_t, \bar\Phi(t), \bar n(t))$  with values in $V\times \R_+^V\times \Z^E$  as follows.
The process $\bar X_t$ is a Markov jump process with jump rates 
$(W_e)_{e\in E}$ (so that $\bar X\stackrel{\text{law}}{=} X$), and
$\bar\Phi(t)$, $\bar n(t)$ are defined by
\begin{eqnarray}\label{tildePhi}
\bar \Phi_x(t)=\sqrt{\bar \Phi_x^2-2\bar \ell_x(t)},
~\forall x\in V,
\end{eqnarray}
where $\bar\ell_x(t)$ is the local time of the process $\check X$ up to time $t$,
\begin{eqnarray}\label{tilden}
\bar n_e(t)= \bar n_e-\left(N_e(J_e(\bar\Phi))-N_e(J_e(\bar\Phi(t)))\right)-\bar K_e(t),
\end{eqnarray}
where $((N_e(v))_{v\ge 0})_{e\in E}$ are independent Poisson point process on $\R_+$ with intensity 1 for each edge $e$, and 
$\bar K_e(t)$ is the number of crossings of the edge $e$ by the process $\bar X$ before time $t$.
We set
\begin{eqnarray}\label{tildeZ}
\bar Z_t=(\bar X_t, (\bar \Phi_x(t)), (\bar n_e(t))).
\end{eqnarray}
This process is well-defined up to time
$$
\bar T=\inf\left\{t\ge 0\vert\exists x\in V~\bar \Phi_x(t)=0\right\}.
$$
We denote by $\bar E_{x_0, \bar\Phi, \bar n}$ its law. Clearly $\bar Z_t=(\bar X_t, \bar\Phi(t), \bar n_e(t))$ is a Markov process, we will later on make explicit its generator. 

We have the following change of variable lemma.
\begin{lemma}\label{change-var}
For all bounded measurable test functions $F,G,H$
\begin{multline*}
\sum_{n\in \N^E} \int d\Phi F(\Phi, n)E_{x_0,\Phi,n}
\left( G((Z_{\tau_u^{x_0}-t})_{0\le t\le\tau_u^{x_0}})
H(\Phi(\tau_u^{x_0}), n(\tau_u^{x_0}))\right)=
\\
\sum_{\bar n\in \N^E} \int d\bar\Phi H(\bar\Phi, \bar n)
\bar E_{x_0,\bar \Phi,\bar n}
\Big(\indic_{\{\bar X_{\bar T}=x_0,~
\forall e\in E,~
\bar n_e(\bar T)\ge 0\}} 
G((\bar Z_{t})_{t\le\bar T})
F(\bar\Phi(\bar T), \bar n(\bar T))\prod_{x\in V\setminus\{x_0\}} {\bar \Phi_x\over \bar\Phi_x(\bar T) }\Big),
\end{multline*}
where the integral on the l.h.s. is on the set 
$\{(\Phi_x)_{x\in V}\in \R_+^V \vert \Phi_{x_0}=0\}$ 
with 
$d\Phi= {\prod_{x\in V\setminus\{x_0\}} 
d\Phi_x\over \sqrt{2\pi}^{\vert V\vert -1}}$
and the integral on the r.h.s. is on the set 
$\{(\bar\Phi_x)_{x\in V}\in \R_+^V \vert\bar\Phi_{x_0}=u\}$ with 
$d\bar\Phi= {\prod_{x\in V\setminus\{x_0\}} d\bar\Phi_x\over \sqrt{2\pi}^{\vert V\vert -1}}$
\end{lemma}
\begin{proof} 
We start from the left-hand side, i.e. the process 
$(X_t, n_e(t))_{0\le t\le \tau_u^{x_0}}$.
We define
$$
\bar X_{t}=X_{\tau_{u}^{x_{0}}-t},\qquad \bar n_e(t)
=n_e(\tau_{u}^{x_{0}}-t),
$$
and
$$
\bar \Phi_x=\Phi_x(\tau_{u}^{x_{0}}),\qquad\bar\Phi_x(t)=\Phi_x({\tau_{u}^{x_{0}}-t}), 
$$
The law of the processes such defined will later be identified with the law of the processes ($\bar X_t, \bar \Phi(t),\bar n(t))$ defined at the beginning of step 2, see \eqref{tildePhi} and \eqref{tilden}.
We also set
$$
\bar K_e(t)= K_e(\tau_{u}^{x_{0}})-K_e(t),
$$
which is also the number of crossings of the edge $e$ by the process $\bar X$, between time 0 and $t$. With these notations we clearly have
$$
\bar \Phi_x(t)=\sqrt{\bar \Phi_x^2-2\bar \ell_x(t)},
$$
where $\bar \ell_x(t)=\int_{0}^t\indic_{\{\bar X_s=x\}} ds$ is the local time of $\bar X$ at time $t$, and
$$
\bar n_e(t)= \bar n_e(0)+(N_e(J_e(\bar \Phi(t)))-N_e(J_e(\bar\Phi(0))))-\bar K_e(t).
$$
By time reversal, the law of 
$(\bar X_t)_{0\le s\le \bar \tau_{u}^{x_{0}}}$ is the same as the law of the Markov Jump process $(X_t)_{0\le t\le \tau_{u}^{x_{0}}}$, where
$\bar \tau_{u}^{x_{0}}=\inf\{t\ge 0\vert\bar\ell_{x_0}(t)=u\}$. Hence, we see that up to the time 
$$\bar T=\inf\{t\ge 0\vert\exists x~\bar\Phi_x(t)=0\},$$ 
the process
$(\bar X_t, (\bar \Phi_x(t))_{x\in V}, (\bar n_e(t))_{e\in E})
_{t\le \bar T}$ 
has the same law as the process defined at the beginning of step 2.

Then, following \cite{SabotTarres2015RK},  we make the following change of variables conditional on the processes $(X_t, (n_e(t))_{e\in E})$:
\begin{eqnarray*}
(0,+\infty)^V\times \N^E&\rightarrow& (0,+\infty)^V\times \N^E\\
((\Phi_x)_{x\in V}, (n_e)_{e\in E})&\mapsto&
((\bar \Phi_x)_{x\in V}, (\bar n_e)_{e\in E}),
\end{eqnarray*}
which is bijective onto the set 
\begin{multline*}
\{(\bar\Phi_x)_{x\in V}\in \R_{+}^{V}\vert\bar\Phi_{x_0}=\sqrt{2u}, ~\forall x\neq x_0,~\check\Phi_x>\sqrt{2\ell_x(\tau_u^{x_0})}\} 
\\\times \{(\bar n_e)_{e\in E}\in \N^{E}\vert
\forall e\in E,~\bar n_e\ge K_e(\tau_{u}^{x_{0}})+(N_e(J_e(\bar \Phi(\tau_{u}^{x_{0}})))-N_e(J_e(\Phi)))\}.
\end{multline*}
Note that we always have $\bar \Phi_{x_0}=\sqrt{2u}$. The last conditions on $\bar \Phi$ and $\bar n_e$ are equivalent to
the conditions $\bar X_{\bar T}=x_0$ and 
$\bar n_e(\bar T)\ge 0$.
The Jacobian of the change of variable is given by
\begin{displaymath}
\prod_{x\in V\setminus\{x_0\}} d\Phi_x=\left({\prod_{x\in V\setminus\{x_0\}} {\bar\Phi_x\over \Phi_x} }\right)\prod_{x\in V\setminus\{x_0\}} d\bar\Phi_x,
\end{displaymath}
since
\begin{displaymath}
d \bar\Phi_{x}=
d \sqrt{\Phi_{x}^{2}+2\ell_{x}(\tau_{u}^{x_{0}})}=
\dfrac{\Phi_{x}}{\sqrt{\Phi_{x}^{2}+2\ell_{x}(\tau_{u}^{x_{0}})}}
~d\Phi_{x}=
\dfrac{\Phi_{x}}{\bar\Phi_{x}} ~d\Phi_{x}.
\qedhere
\end{displaymath}
\end{proof}

\noindent
{\bf Step 3:}
With the notations of Theorem \ref{thm-Poisson2}, we consider the following expectation for $g$ and $h$ bounded measurable test functions:
\begin{eqnarray}\label{test-functions}
\E\left( g\left(\left(X_{\tau_{u}^{x_{0}}-t}, n_e(\tau_{u}^{x_{0}}-t)\right)_{0\le t\le \tau_{u}^{x_{0}}}\right)h(\varphi^{(u)})\right).
\end{eqnarray}
By definition, we have
$$
\varphi^{(u)}=\sigma \Phi(\tau_{u}^{x_{0}}),
$$
where $(\sigma_x)_{x\in V}\in \{-1,+1\}^V$ are random signs sampled uniformly independently on clusters induced by
$\{e\in E\vert n_e(\tau_{u}^{x_{0}})>0\}$ and conditioned on 
$\sigma_{x_0}=+1$.
Hence, we define for $(\Phi_x)_{x\in V}\in \R_+^V$ and 
$(n_e)_{e\in E}\in \N^E$,
\begin{eqnarray}\label{h}
H(\Phi, n)=2^{-\#c.\ccc(n)+1} 
\sum_{\substack{\sigma\in\{\pm 1\}^{V}\\
\sigma_{x_{0}}=+1,~\sigma\ll n}} h(\sigma \Phi),
\end{eqnarray}
where $\sigma\ll n$ means that the signs $\sigma_x$ are constant on clusters of $\ccc(n)$.
%$\{ e\in E\vert n_e>0\}$ 
%and such that $\sigma_{x_0}=+1$.
Hence, setting
$$
F(\Phi, n)=e^{-\demi \sum_{x\in V} W_x \Phi_x^2-\sum_{e\in E} J_e(\Phi) }\left(\prod_{e\in E} {(2J_e(\Phi))^{n_e}\over n_e!}\right)
2^{\#c.\ccc(n)-1},
$$
$$
G\left((Z_{\tau_{u}^{x_{0}}-t})_{t\le\tau_{u}^{x_{0}}}\right)= g\left(\left(X_{\tau_{u}^{x_{0}}-t}, n_e(\tau_{u}^{x_{0}}-t)\right)_{t\le \tau_{u}^{x_{0}}}\right),
$$
using Lemma \ref{distrib-phi-n} in the first equality and Lemma \ref{change-var} in the second equality, we deduce that
\eqref{test-functions} is equal to 
\begin{multline}
\label{eq-3.3}
\E\left( G\left((Z_{\tau_{u}^{x_{0}}-t})_{0\le t\le \tau_{u}^{x_{0}}}\right)H(\Phi(\tau_{u}^{x_{0}}), n(\tau_{u}^{x_{0}})))\right)=
\\
\sum_{n\in \N^E} \int
%_{\R_+^{V\setminus\{x_0\}}}  
d\Phi
F(\Phi, n) E_{x_0, \Phi,n}\left(G\left((Z_{\tau_{u}^{x_{0}}-t})_{t\le\tau_{u}^{x_{0}}}\right)H\left(\Phi(\tau_{u}^{x_{0}}), 
n(\tau_{u}^{x_{0}}))\right)\right)
d\Phi =
\\
\sum_{\bar n\in \N^E} \int d\bar\Phi
%_{\R_+^{V\setminus\{x_0\}}}  
H\left(\bar \Phi,\bar n\right)
\bar E_{x_0, \bar \Phi, \bar n}\Big(\indic_{\{\bar X_{\bar T}=x_0,~\forall e\in E~\bar n_e(\bar T)\ge 0\}} 
F\left(\bar \Phi(\bar T) , \bar n(\bar T)\right) G\left((\bar Z_{t})_{t\le\bar T}\right) \prod_{x\in V\setminus\{x_0\}} {\bar \Phi_x\over \bar\Phi_x(\bar T) }
\Big),
% {d\Phi\over \sqrt{2\pi}^{\vert V\vert -1}}
\end{multline}
with notations of Lemma \ref{change-var}.
%where we note $d\Phi={\prod_{x\in V\setminus\{x_0\}} d\Phi_x\over \sqrt{2\pi}^{\vert V\vert -1}}$.

Let $\bar\fff_t=\sigma((\bar X_s)_{s\le t})$ be the filtration generated by $\bar X$. We define the $\bar \fff$-adapted process
$\bar M_t$, defined up to time $\bar T$ by
\begin{multline}
\label{Mart}
\bar M_t
= {F(\bar \Phi(t), \bar n(t))\over \prod_{x\in V\setminus\{\bar X_t\}} \bar\Phi_x(t) }
\indic_{\{\bar X_t\in \ccc(x_0,\bar n)\}}
\indic_{\{\bar n_e(t)\ge 0,~\forall e\in E\}}=
e^{-\demi \sum_{x\in V} W_x\bar \Phi_x(t)^2-\sum_{e\in E} J_e(\bar\Phi(t)) }
%+ \demi \sum_{x\in V} (\bar \Phi_x)^2
\times
\\\times
\Big(\prod_{e\in E} {(2J_e(\bar \Phi(t)))^{\bar n_e(t)}\over \bar n_e(t) !}\Big)
%\left( {\prod_{e\in E} (2J_e(\bar \Phi))^{\bar n_e}\over \bar n_e !}\right)^{-1}
{2^{\#c.\ccc(\bar n(t))-1}
\over \prod_{x\in V\setminus\{\bar X_t\}} \bar\Phi_x(t) }
\indic_{\{\bar X_t\in \ccc(x_0,\bar n(t))\}}
\indic_{\{\bar n_e(t)\ge 0 , ~\forall e\in E\}},
%\ccc(\bar n)}
\end{multline}
where $\ccc(x_0,\bar n(t))$ denotes the cluster of the origin $x_0$ induced by the configuration $\ccc(\bar n(t))$.
Note that at time $t=\bar T$, we also have
\begin{eqnarray}\label{M-T}
\bar M_{\bar T}= {F(\bar \Phi(\bar T), \bar n(\bar T))\over \prod_{x\in V\setminus\{x_0\}} \bar\Phi_x(\bar T) }
\indic_{\{\bar X_{\bar T}=x_0\}}
\indic_{\{\bar n_e(t)\ge 0,~\forall e\in E\}} 
\end{eqnarray}
since $\bar M_{\bar T}$ vanishes on the event where $\{\bar X_{\bar T}=x\}$, with $x\neq x_0$. Indeed, if $\bar X_{\bar T}=x\neq x_0$, then
$\bar\Phi_x(\bar T)=0$ and $J_e(\bar\Phi(\bar T))=0$ for $e\in E$ such that $e$ adjacent to $x$. 
It means that $\bar M_{\bar T}$ is equal to
0 if $\bar n_{e}(\bar T)>0$ for some edge $e$ neighboring $x$. Thus, $\bar M_{\bar T}$ is null unless $\{x\}$ is a cluster in 
$\ccc(\bar n(\bar T))$.
Hence, $\bar M_{\bar T}=0$ if $x\neq x_0$ since $\bar M_{\bar T}$ contains the indicator of the event that $\bar X_{\bar T}$ and $x_0$ are in the same cluster. 

Hence, using identities \eqref{eq-3.3} and \eqref{M-T}
we deduce that \eqref{test-functions} is equal to 
\begin{eqnarray}
\label{equ-M}
(\ref{test-functions})&=&
\sum_{\bar n\in \N^E} \int d\bar\Phi
%_{\R_+^{V\setminus\{x_0\}}}  
H\left(\bar \Phi,\bar n\right) F\left(\bar \Phi,\bar n\right)
\bar E_{x_0, \bar \Phi, \bar n}\left(
{\bar M_{\bar T}\over \bar M_0} 
 G\left((\bar Z_{t})_{t\le\bar T}\right)
\right).
\end{eqnarray}

\noindent {\bf Step 4 :}
We denote by $\check Z_t=(\check X_t, \check \Phi_t, \check n(t))$ the process defined in Section \ref{sec_Poisson}, which is well defined up to stopping time $\check T$, and $\check Z^T_t=\check Z_{t\wedge \check T}$. We denote by $\check E_{x_0, \check \Phi, \check n}$
the law of the process $\check Z$ conditional on the initial value $\check n(0)$, i.e. conditional on $(N_e(2J(\check\Phi)))=(\check n_e)$.
The last step of the proof goes through the following lemma.
\begin{lemma}\label{RN}
i) Under $\check E_{x_0,\check\Phi,\check n}$, $\check X$ ends at $\check X_{\check T}=x_0$ a.s. and 
$\check n_e(\check T)\ge 0$ for all $e\in E$.

ii) Let $\bar P^{\le t}_{x_0,\bar\Phi,\bar n}$ and $\check P^{\le t}_{x_0,\check\Phi,\check n}$ be the law
of the process $(\bar Z^T_s)_{s\le t}$ and $(\check Z^T_s)_{s\le t}$
respectively, then
$$
{d\check P^{\le t}_{x_0,\bar \Phi,\bar n}\over d\bar P^{\le t}_{x_0,\bar \Phi,\check n}}={\bar M_{t\wedge \bar T}\over \bar M_0}.
$$
\end{lemma}
Using this lemma we obtain that in the right-hand side of \eqref{equ-M}
$$
\bar E_{x_0, \bar \Phi , \bar n}\left(
{\bar M_{\bar T}\over \bar M_0} 
 G\left((\bar Z_{t})_{t\le\bar T}\right)\right)=
 \check E_{x_0, \bar \Phi , \bar n}
 \left(  
 G\left((\check Z_{t})_{t\le\check T}\right)\right).
$$
Hence, we deduce, using formula \eqref{h} and proceeding as in Lemma \ref{distrib-phi-n}, that \eqref{test-functions} is equal to
\begin{multline*}
\label{final}
\int_{\R^{V\setminus\{x_0\} }} d\bar\varphi
%_{\R_+^{V\setminus\{x_0\}}}  
e^{-\demi \eee(\bar\varphi,\bar\varphi)}  h(\bar \varphi) 
\sum_{\bar n\ll \bar \varphi} \left(\prod_{e\in E, \; \bar\varphi_{e_-}\bar\varphi_{e_+}\ge 0} 
{e^{-2J_e(\vert \bar \varphi\vert)}(2J_e(\vert \bar \varphi\vert ))^{\bar n_e}\over \bar n_e !}\right) 
\\\bar E_{x_0, \vert \bar \varphi\vert , \bar n}\left({\bar M_{\bar T}\over \bar M_0} 
 G\left((\bar Z_{t})_{t\le\bar T}\right)\right),
\end{multline*}
where the last integral is on the set 
$\{(\bar\varphi_x)_{x\in V}\in \R^V\vert \varphi_{x_0}=u\}$,
$d\bar\varphi={\prod_{x\in V\setminus\{x_0\}} d\bar\varphi_x\over \sqrt{2\pi}^{\vert V\vert -1}}$, and where 
$\bar n\ll \bar\varphi$ means that
%$\bar n\in \N^E$ and 
$\bar n_e=0$ if $\bar\varphi_{e_-}\bar\varphi_{e_+}\le 0$.
Finally, we conclude that
\begin{eqnarray*}
\E\left[ g\left(\left(X_{\tau_u^{x_0}-t}, n_e(\tau_u^{x_0}-t)\right)_{0\le t\le \tau_u^{x_0}}\right)h(\varphi^{(u)})\right]=
\E\left[ g\left(\left(\check X_{t}, \check n_e(t)\right)_{0\le t\le \check T}\right)h(\check \varphi)\right],
\end{eqnarray*}
where in the right-hand side 
$\check \varphi\sim P_{\varphi}^{\{x_0\}, \sqrt{2u}} $ 
is a GFF and $(\check X_t, \check n(t))$ 
is the process defined in Section \ref{sec_Poisson} from the
GFF $\check \varphi$.
This exactly means that 
$\varphi^{(u)} \sim P_{\varphi}^{\{x_0\}, \sqrt{2u}}$ 
and that
$$
\text{Law}\left(\left(X_{\tau_u^{x_0}-t}, n_e(\tau_u^{x_0}-t)\right)_{0\le t\le \tau_u^{x_0}}\; 
\Big| \; \varphi^{(u)}=\check\varphi\right)= 
\text{Law}\left(\left(\check X_t, \check n(t)\right)_{t\le \check T}\right).
$$
This concludes the proof of Theorem \ref{thm-Poisson2}.
\end{proof}
\begin{proof}[Proof of Lemma \ref{RN}]
The generator of the process $\bar Z_t$ defined in (\ref{tildeZ}) is given, for any bounded and $\mathcal{C}^{1}$ for the second component test function $f$, by 
\begin{equation}
\label{tildeL2}
\begin{split}
&(\bar{\mathfrak{L}} f)(x,\bar\Phi,\bar n)=
-{1\over \bar \Phi_x} ({\partial\over \partial \bar\Phi_x}f)
(x,\bar\Phi, \bar n) +\\
&
\sum_{\substack{y\in V\\y\sim x}} \left(W_{x,y} \left(f(y,\bar\Phi,\bar n-\delta_{\{x,y\}})-f(x,\bar\Phi,n)\right)+
W_{x,y} {\bar \Phi_{y}\over \bar \Phi_x} \left(f(x,\bar\Phi, n-\delta_{\{x,y\}})-f(x,\bar\Phi,n)\right)\right).
\end{split}
\end{equation}
where $ n-\delta_{\{x,y\}}$ is the value obtained by removing 1 from $n$ at edge $\{x,y\}$.
Indeed, since $\bar \Phi_x(t)=
\sqrt{\bar\Phi_{x}(0)^{2} -2\bar \ell_x(t)}$, we have 
\begin{eqnarray}
\label{deriv-Phi}
{d\over d t} \bar \Phi_x(t)=
-\indic_{\{\bar X_t=x\}}{1\over \bar \Phi_x(t)},
\end{eqnarray}
which explains the first term in the expression. The second term is obvious from the definition of $\bar Z_t$, and corresponding to the term induced by jumps
of the Markov process $\bar X_t$. The last term corresponds to the decrease of $\bar n$ due to the increase in the process
$N_e(J_{e}(\bar \Phi))-N_e(J_{e}(\bar \Phi(t)))$. Indeed, on the interval $[t,t+dt]$, the probability that
$N_{e}(J_{e}(\bar \Phi(t)))-N_{e}(J_{e}(\bar \Phi(t+dt)))$
is equal to 1 is of order
$$-{d\over d t} N_e(J_{e}(\bar \Phi(t)))dt=
\indic_{\{\bar X_t~\text{endpoint of}~e\}}
 {W_e \bar \Phi_{e_{-}}(t)\bar\Phi_{e_{+}}(t)\over \Phi_{\bar X_t}(t)^2}dt,
 $$
using identity \eqref{deriv-Phi}.

Let $\check{\mathfrak{L}}$ be the generator of the Markov jump process $\check Z_t=(\check X_t, (\check \Phi_x(t)), (\check n_e(t)))$.
We have that the generator is equal, for any smooth test function $f$, to
 % Under $\check E_{x_0, \check \Phi, \check n}$, $\check Z$ is a Markov process with generator
\begin{eqnarray}
\label{EqCheckL}
&&(\check{\mathfrak{L}} f)(x,\Phi, n)=
-{1\over  \Phi_x} ({\partial\over \partial \Phi_x}f)(x,\Phi,  n) +\\
\nonumber
&&\demi \sum_{\substack{y\in V\\y\sim x}}{ n_{x,y} \over  \Phi_x^2}
{\indic_{\aaa_1(x,y)}} \left(f(y,\bar\Phi,n-\delta_{\{x,y\}})+f(x,\bar\Phi,n-\delta_{\{x,y\}})- 2f(x,\bar\Phi,n)\right)
\\
\nonumber
&&+ \sum_{\substack{y\in V\\y\sim x}}{ n_{x,y} \over \Phi_x^2}  \indic_{\aaa_2(x,y)} \left( f(y,\bar\Phi,n-\delta_{\{x,y\}})- f(x,\bar\Phi,n)) \right)
\\
\nonumber
&&+\sum_{\substack{y\in V\\y\sim x}}{n_{x,y} \over \Phi_x^2} 
 \indic_{\aaa_3(x,y)} \left(f(x,\bar\Phi,n-\delta_{\{x,y\}}) - f(x,\bar\Phi,n) \right),
\end{eqnarray}
where 
$\aaa_{i}(x,y)$ correspond to the following disjoint events:
\begin{itemize}
\item
$\aaa_1(x,y)$ if the numbers of connected clusters induced by $n-\delta_{\{x,y\}}$ is the same as that of $\check n$;
\item
$\aaa_2(x,y)$ if a new cluster is created in $ n-\delta_{\{x,y\}}$ compared with $\check n$ and if $y$ is in the connected component
of $x_0$ in the cluster induced by $ n-\delta_{\{x,y\}}$;
\item
$\aaa_3(x,y)$ if a new cluster is created in $ n-\delta_{\{x,y\}}$ compared with $n$ and if $x$ is in the connected component
of $x_0$ in the cluster induced by $ n-\delta_{\{x,y\}}$.
\end{itemize}
Indeed, conditional on the value of $\check n_e(t)=N_e(2J_e(\check\Phi(t)))$ at time $t$, the point process $N_e$ on the interval 
$[0,  2J_e(\check\Phi(t))]$ has the law of
$n_e(t)$ independent points with uniform distribution on 
$[0,  2J_e(\check\Phi(t))]$. Hence, the probability that a point lies in the interval 
$[2J_e(\check\Phi(t+dt)),  2J_e(\check\Phi(t))]$ is of order 
$$
-\check n_e(t) {1\over J_e(\check\Phi(t))}{d\over d t} J_e(\check\Phi(t))  dt= \indic_{\{X_t~\text{endpoint of}~e\}}\;\check n_e(t){1\over \check\Phi_{X_t}(t)^2}dt.
$$
We define the function 
\begin{multline}
%\label{Theta}
\nonumber\Theta(x,(\Phi_x),(n_e))=\\
e^{-\demi \sum_{x\in V}W_x \Phi_x^2-\sum_{e\in E} J_e(\Phi) }
%+ \demi \sum_{x\in V} (\bar \Phi_x)^2
\left(\prod_{e\in E} {(2J_e(\Phi))^{n_e}\over n_e !}\right)
%\left( {\prod_{e\in E} (2J_e(\bar \Phi))^{\bar n_e}\over \bar n_e !}\right)^{-1}
{2^{\#c.\ccc(n)-1}
\over \prod_{y\in V\setminus\{x\}} \Phi_{y} }
\indic_{\{x\in \ccc(x_0,n),~\text{and}~\forall e\in E,~
n_e\ge 0\}},
\end{multline}
so that
$$
\bar M_{t\wedge \bar T}= \Theta(\bar Z_{t\wedge\bar T}).
$$
To prove the lemma it is sufficient to prove (\cite{ChungWalsh05MP}, Chapter 11) that for any bounded smooth test function $f$,
\begin{eqnarray}\label{LcheckL}
{1\over \Theta}\bar{\mathfrak{L}}\left(\Theta f\right)= \check{\mathfrak{L}}\left(f\right).
 \end{eqnarray}
Let us first consider the first term in (\ref{tildeL2}).
Direct computation gives 
$$
\left({1\over \Theta}{1\over \Phi_x}\left({\partial\over\partial \Phi_x} \Theta\right)\right) (x,\Phi,n)= -W_x
+\sum_{\substack{y\in V\\y\sim x}} \left(- W_{x,y}{\Phi_y\over\Phi_x}+n_{x,y}{1\over \Phi_x^2}\right).
$$
For the second part, remark that the indicators $\indic_{\{x\in \ccc(x_0,n)\}}$ and $\indic_{\{n_e\ge 0,~ \forall e\in E\}}$ imply that 
$
\Theta(y,\Phi, n-\delta_{\{x,y\}})
$
vanishes if $n_{x,y}=0$ or if $y\not\in \ccc(x_0,n-\delta_{\{x,y\}})$.
By inspection of the expression of $\Theta$, we obtain for $x\sim y$,
\begin{eqnarray*}
\Theta (y,\Phi, n-\delta_{\{x,y\}})&=& \left(\indic_{\{n_{x,y}>0\}} (\indic_{\aaa_1}+2\indic_{\aaa_2}) {n_{x,y}\over 2J_{x,y}(\Phi)}{\Phi_y\over \Phi_x}\right)\Theta(x,\Phi, n) 
\\
&=&\left((\indic_{\aaa_1}+2\indic_{\aaa_2}) {n_{x,y}\over 2W_{x,y}}{1\over \Phi_x^2}\right)\Theta(x,\Phi, n).
\end{eqnarray*}
Similarly,  for $x\sim y$,
\begin{eqnarray*}
\Theta(x,\Phi, n-\delta_{\{x,y\}})&=& \left(\indic_{\{n_{x,y}>0\}}(\indic_{\aaa_1}+2\indic_{\aaa_3}){n_{x,y}\over 2J_{x,y}}\right)\Theta(x,\Phi, n)\\
&=&
\left((\indic_{\aaa_1}+2\indic_{\aaa_3}) {n_{x,y}\over 2W_{x,y}\Phi_x\Phi_y}\right)\Theta(x,\Phi, n).
\end{eqnarray*}
Combining these three identities with the expression \eqref{tildeL2} we deduce
\begin{eqnarray*}
&&{1\over \Theta}\bar{\mathfrak{L}}\left(\Theta f\right)(x,\Phi,n)=\\
&&
-{1\over \Phi_x} {\partial\over\partial \Phi_x}f(x,\Phi,n)
-\sum_{\substack{y\in V\\y\sim x}} \left(n_{x,y}{1\over \Phi_x^2}\right)f(x,\Phi,n)
\\
&&  +\sum_{\substack{y\in V\\y\sim x}} (\indic_{\aaa_1}+2\indic_{\aaa_2}) n_{x,y}{1\over 2\Phi_x^2} f(y,  n-\delta_{\{x,y\}},\Phi)+
\sum_{\substack{y\in V\\y\sim x}}(\indic_{\aaa_1}+2\indic_{\aaa_3}){1\over 2 \Phi_x^2} f(x, n-\delta_{\{x,y\}},\Phi).
\end{eqnarray*}
It exactly coincides with the expression 
\eqref{EqCheckL}
for $\check{\mathfrak{L}}$ since 
$1=\indic_{\aaa_1}+\indic_{\aaa_2}+\indic_{\aaa_3}$.
\end{proof}

\subsection{General case}
\label{sec:pgen}

\begin{proposition}
\label{PropKillingCase}
The conclusion of
Theorem \ref{thm-Poisson} still holds
if the graph $\mathcal{G}=(V,E)$ is finite and the killing measure is non-zero ($\kappa\not\equiv 0$).
\end{proposition}

\begin{proof}
Let $h$ be the function on $V$ defined as
\begin{displaymath}
h(x)=\mathbb{P}_{x}(X~\text{hits}~x_{0}~\text{before}~\zeta).
\end{displaymath}
By definition $h(x_{0})=1$. Moreover, for all 
$x\in V\setminus\lbrace x_{0}\rbrace$,
\begin{displaymath}
-\kappa_{x} h(x)+\sum_{\substack{y\in V\\y\sim x}}W_{x,y}(h(y)-h(x))=0.
\end{displaymath}
Define the conductances
$W^{h}_{e}:=W_{e}h(e_{-})h(e_{+})$, and the corresponding jump process $X^{h}$, and the GFF $\varphi_{h}^{(0)}$ and $\varphi_{h}^{(u)}$ with conditions
$0$ respectively $\sqrt{2u}$ at $x_{0}$. The Theorem \ref{thm-Poisson}
holds for the graph $\mathcal{G}$ with conductances 
$(W^{h}_{e})_{e\in E}$ and with zero killing measure.
Denote
\begin{displaymath}
\ell_{x}^{h}(t)=\int_{0}^{t}\indic_{\{X^{h}_{s}=x\}}ds,\qquad
\tau_{u}^{x,h}=\inf\lbrace t\geq 0\vert \ell_{x_{0}}^{h}(t)\geq \sqrt{2u}\rbrace.
\end{displaymath} 
The process
$(X^{h}_{t})_{t\leq \tau_{u}^{x_{0},h}}$ 
has the same law as the process
$(X_{\theta^{h}(t)})_{t\leq (\theta^{h})^{-1}(\tau_{u}^{x_{0}})}$, conditional on
$\tau_{u}^{x_{0}}<\zeta$, after the change of time
\begin{displaymath}
d\theta^{h}(t)= h(X_{\theta^{h}(t)})^{2}dt.
\end{displaymath}
%\begin{displaymath}
%dt = h(X_{s})^{-2}ds.
%\end{displaymath}
This means in particular that for the occupation times,
\begin{equation}
\label{EqTimeChange}
\ell_{x}^{h}(t)=h(X_{\theta^{h}(t)})^{-2}\ell_{x}(\theta^{h}(t)).
\end{equation}
Moreover, we have the equalities in law 
\begin{displaymath}
\varphi_{h}^{(0)}\stackrel{\text{law}}{=}h^{-1}\varphi^{(0)},\qquad
\varphi_{h}^{(u)}\stackrel{\text{law}}{=}h^{-1}\varphi^{(u)}.
\end{displaymath}
Indeed, at the level of energy functions, we have:
\begin{equation*}
\begin{split}
&\mathcal{E}(hf,hf)=
\sum_{x\in V}\kappa_{x} h(x)^{2}f(x)^{2}+
\sum_{e\in E}W_{e}(h(e_{+})f(e_{+})-h(e_{-})f(e_{-}))^{2}\\&=
\sum_{x\in V}[\kappa_{x}h(x)^{2}f(x)^{2}+
\sum_{\substack{y\in V\\y\sim x}}W_{x,y}h(y)f(y)(h(y)f(y)-h(x)f(x))]\\
&=
\sum_{x\in V}[\kappa_{x}h(x)^{2}f(x)^{2}-
\sum_{y\sim x}W_{x,y}(h(y)-h(x))h(x)f(x)^{2}]
-\sum_{\substack{x\in V\\y\sim x}}W_{x,y}h(x)h(y)(f(y)-f(x))f(x)
\\&=[\kappa_{x_{0}}-
\sum_{\substack{y\in V\\y\sim x_{0}}}W_{x_{0},y}(h(y)-1)]f(x_{0})^{2}
+\sum_{e\in E}W_{e}^{h}(h(e_{+})f(e_{+})-h(e_{-})f(e_{-}))^{2}
\\&= \text{Cst}(f(x_{0}))+\mathcal{E}^{h}(f,f),
\end{split}
\end{equation*}
where $\text{Cst}(f(x_{0}))$ means that this term does not depend of $f$
once the value of the function at $x_{0}$ fixed.

Let $\check{X}^{h}_{t}$ be the inverse process for the conductances 
$(W_{e}^{h})_{e\in E}$ and the initial condition for the field
$\varphi_{h}^{(u)}$, given by Theorem \ref{thm-Poisson}. 
By applying the inverse of the time change
\eqref{EqTimeChange} to the process $\check{X}^{h}_{t}$, we obtain an inverse process for the conductances $W_{e}$ and the field 
$\varphi^{(u)}$.
\end{proof}

\begin{proposition}
\label{PropInfiniteCase}
Assume that the graph $\mathcal{G}=(V,E)$ is infinite. The killing measure $\kappa$ may be non-zero. Then the conclusion of
Theorem \ref{thm-Poisson} holds.
\end{proposition}

\begin{proof}
Consider an increasing sequence of connected sub-graphs
$\mathcal{G}_{i}=(V_{i},E_{i})$ of $\mathcal{G}$ which converges to the whole graph. We assume that $V_{0}$ contains $x_{0}$.
Let $\mathcal{G}_{i}^{\ast}=(V_{i}^{\ast},E_{i}^{\ast})$ be the graph obtained by adding to $\mathcal{G}_{i}$ an abstract vertex
$x_{\ast}$, and for every vertex $x\in V_{i}$ connected by an edge in $E_{i}$ to a $y\in V\setminus V_{i}$, adding an edge 
$\lbrace x,x_{\ast}\rbrace$
with with a conductance
$$W_{x,x_{\ast}}=
\sum_{\substack{y\in V\setminus V_{i}\\y\sim x}}
W_{x,y}.$$
$(X_{i,t})_{t\geq 0}$ will denote the Markov jump process on 
$\mathcal{G}_{i}^{\ast}$, started from $x_{0}$. 
Let $\zeta_{i}$ be the first hitting time of $x_{\ast}$ or the first
killing time by the measure $\kappa\indic_{V_{i}}$. Let
$\varphi^{(0)}_{i}$,
$\varphi^{(u)}_{i}$ will denote the GFFs on $\mathcal{G}_{i}^{\ast}$ with condition $0$ respectively $\sqrt{2u}$ at $x_{0}$, with condition $0$ at 
$x_{\ast}$, and taking in account the possible killing measure 
$\kappa\indic_{V_{i}}$. 
The limits in law of $\varphi^{(0)}_{i}$
respectively $\varphi^{(u)}_{i}$ are 
$\varphi^{(0)}$
respectively $\varphi^{(u)}$.

We consider the process 
$(\check{X}_{i,t},(\check{n}_{i,e}(t))_{e\in E_{i}^{\ast}})
_{0\leq t\leq\check{T}_{i}}$ be the inverse process on
$\mathcal{G}_{i}^{\ast}$, with initial field $\varphi^{(u)}_{i}$.
$(X_{i,t})_{t\leq \tau_{i,u}^{x_{0}}}$, conditional on 
$\tau_{i,u}^{x_{0}}$, has the same law as
$(\check{X}_{i,\check{T}_{i}-t})_{t\leq \check{T}_{i}}$.
Taking the limit in law as $i$ tends to infinity, we conclude that
$(X_{t})_{t\leq \tau_{u}^{x_{0}}}$, conditional on 
$\tau_{u}^{x_{0}}<+\infty$, has the same law as
$(\check{X}_{\check{T}-t})_{t\leq \check{T}}$ on the infinite graph
$\mathcal{G}$. The same for the clusters.
In particular,
\begin{multline*}
\mathbb{P}(\check{T}\leq t, \check{X}_{[0,\check{T}]}~\text{stays in}~V_{j})\geq
\lim_{i\to +\infty}
\mathbb{P}(\check{T}_{i}\leq t, \check{X}_{i,[0,\check{T}_{i}]}~\text{stays in}~V_{j})
\\=
\lim_{i\to +\infty}
\mathbb{P}(\tau_{i,u}^{x_{0}}\leq t, X_{i,[0,\tau_{i,u}^{x_{0}}]}~\text{stays in}~V_{j}\vert \tau_{i,u}^{x_{0}}<\zeta_{i})=
\mathbb{P}(\tau_{u}^{x_{0}}\leq t, X_{[0,\tau_{u}^{x_{0}}]}
~\text{stays in}~V_{j}\vert \tau_{u}^{x_{0}} < \zeta),
\end{multline*}
where in the first two probabilities we also average by the values of the
free fields.
Hence
\begin{displaymath}
\mathbb{P}(\check{T}=+\infty~\text{or}~\check{X}_{\check{T}}\neq x_{0})=
1-\lim_{\substack{t\to +\infty\\ j\to +\infty}}
\mathbb{P}(\tau_{u}^{x_{0}}\leq t, X_{[0,\tau_{u}^{x_{0}}]}
~\text{stays in}~V_{j}\vert \tau_{u}^{x_{0}} < \zeta) = 0.
\qedhere
\end{displaymath}
\end{proof}

\begin{remark}
\label{RemInterlacement}
Consider $\mathcal{G}=(V,E)$ an infinite transient electrical network
(with $\kappa\equiv 0$). Proposition \ref{PropInfiniteCase} tells that if the inversions algorithm of Section \ref{sec:inversion} is applied to a Gaussian free field 
$\varphi^{(u)}$ with condition $\sqrt{2u}$ at $x_{0}$, and implicitly $0$ at infinity, the algorithm terminates a.s., that is to say the inverting process
$\check{X}$ does not escape to infinity. However, one could consider a Gaussian free field with positive condition $a>0$ at infinity, 
$\varphi^{(u,a)}$. Such a GFF is related by isomorphism not only to a loop-soup 
$\mathcal{L}_{1/2}$ but also to a Sznitman's random interlacement,
which is a Poisson point process of paths from and to infinity, infinite in both directions of time
\cite{Sznitman2012Isomorphism,Sznitman2012LectureIso,Lupu2014LoopsGFF}. If applied to $\varphi^{(u,a)}$, the algorithm would create a path which has a positive probability to escape to infinity, which would correspond to the event of having an interlacement visiting $x_{0}$.
\end{remark}

\section*{Acknowledgements}

This work was supported by the French National Research Agency (ANR) grant
within the project MALIN (ANR-16-CE93-0003).

This work was partly supported by the LABEX MILYON (ANR-10-LABX-0070) of Université de Lyon, within the program "Investissements d'Avenir" (ANR-11-IDEX-0007) operated by the French National Research Agency (ANR).

TL acknowledges the support of Dr. Max Rössler, the Walter Haefner
Foundation and the ETH Zurich Foundation.

The authors would like to thank the anonymous referee for his comments on the previous version of this paper.

\bibliographystyle{plain}
\bibliography{ray-knight}
\end{document}